\documentclass[a4paper,12pt]{article}

\usepackage{fullpage}	
\usepackage{amsmath, amssymb, amsthm}

\usepackage[latin1]{inputenc}
\usepackage{xy}
\usepackage{graphicx}
\usepackage{verbatim}
\usepackage{enumerate}
\usepackage{url}
\usepackage{subfig}
\usepackage{float}
\usepackage{todonotes}

\usepackage{hyperref}
\usepackage{enumerate}

\newcommand{\eps}{\varepsilon}
\newcommand{\Ob}{\mathcal{O}}

\newcommand{\bbR}{\mathbb{R}}
\newcommand{\bbS}{\mathbb{S}}
\newcommand{\bbT}{\mathbb{T}}
\newcommand{\bbZ}{\mathbb{Z}}
\newcommand{\Z}{\bbZ}

\newcommand{\rg}{\mathrm{g}}

\newcommand{\cC}{\mathcal{C}}
\newcommand{\CIR}{\cC}
\newcommand{\cF}{\mathcal{F}}
\newcommand{\cG}{\mathcal{G}}

\newcommand{\cH}{\mathcal{H}}

\newcommand{\CZ}{\mathrm{LC}}
\newcommand{\LZ}{\mathrm{LZ}}
\newcommand{\DZ}{\mathrm{DZ}}

\newcommand{\bes}{\boldsymbol{\emptyset}}

\newcommand{\bfA}{\mathbf{A}}
\newcommand{\bfC}{\mathbf{C}}

\newcommand{\bfG}{\mathbf{G}}
\newcommand{\bfH}{\mathbf{H}}
\newcommand{\bfK}{\mathbf{K}}
\newcommand{\bfP}{\mathbf{P}}

\newcommand{\pugraph}{\bar\bfP}
\newcommand{\cugraph}{\bar\bfC}

\newcommand{\bfn}{\mathbf{n}}

\newcommand{\egraph}{\bes}

\newcommand{\G}{\Gamma}
\newcommand{\g}{\gamma}

\newcommand{\Sync}{\Theta^\mathrm{sync}}
\newcommand{\Splay}{\Theta^\mathrm{splay}}

\newcommand{\rset}[2]{\left\lbrace\, #1\,\left|\;#2\right.\right\rbrace}

\newcommand{\set}[2]{\rset{#1}{#2}}
\newcommand{\sset}[1]{\left\lbrace #1\right\rbrace}

\newcommand{\tabs}[1]{\big|#1\big|}

\DeclareMathOperator{\inte}{int}

\DeclareMathOperator{\fix}{Fix}

\newtheorem{prop}{Proposition}
\newtheorem{lem}{Lemma}

\newtheorem{cor}{Corollary}
\theoremstyle{definition}
\newtheorem{defn}{Definition}

\theoremstyle{remark}
\newtheorem{rem}{Remark}

\author{Peter Ashwin, Christian Bick, Camille Poignard\\
EPSRC Centre for Predictive Modelling in Healthcare\\
and Centre for Systems, Dynamics and Control,\\
Department of Mathematics,\\
University of Exeter, Exeter EX4 4QF, UK}

\title{State-dependent effective interactions in oscillator networks through coupling functions with dead zones}

\begin{document}

\maketitle

\begin{abstract}
The dynamics of networks of interacting dynamical systems depend on the nature of the coupling between individual units. We explore networks of oscillatory units with coupling functions that have ``dead zones'', that is, the coupling functions are zero on sets with interior. For such networks, it is convenient to look at the effective interactions between units rather than the (fixed) structural connectivity to understand the network dynamics. For example, oscillators may effectively decouple in particular phase configurations. Along trajectories the effective interactions are not necessarily static, but the effective coupling may evolve in time. Here, we formalize the concepts of dead zones and effective interactions. We elucidate how the coupling function shapes the possible effective interaction schemes and how they evolve in time.
\end{abstract}

\maketitle
\allowdisplaybreaks

\newcommand{\TRt}{\T^2\times\bbR^2}
\newcommand{\avOb}{\langle\Ob\rangle}
\newcommand{\lsOb}{\limsup\Ob}

\newcommand{\sm}{\smallsetminus}
\newcommand{\abs}[1]{\left|#1\right|}

\section{Introduction}
\label{sec:intro}

Many systems in applied sciences can be seen as systems of coupled units that mutually influence each other, such as interacting neurons of an animal's nervous system. Moreover, the system's functionality often depends on emergent properties of this dynamical system. The dynamical behaviour of a network composed of coupled systems depends on a number of factors. One might wish to separate these factors into three broad categories: Firstly, the dynamics of component systems (the \emph{unit dynamics}) in isolation, secondly the structure of the graph of couplings between the systems (the \emph{network structure} viewed as a directed graph where the nodes are systems and the edges are connections) and thirdly the nature of the interactions (we refer in general to this as \emph{coupling functions}~\cite{Staetal2017}). 

This approach can be limiting in several ways. First, it is well-known that coupling isolated systems with simple unit dynamics may result in relatively simple dynamics (such as synchronization), equally it can result in emergence of qualitatively different dynamics~\cite{Pikovsky01}. While a graph structure---as a graph of connections---is an efficient way to encode linearly weighted coupling between dynamical units, it may not capture higher-order, multi-way coupling. Such interactions are typically not ``pairwise'' but there are nonlinear interactions between three or more nodes, for example, if the input from unit~2 to unit~1 is modulated by unit~3.
This leads to more general \emph{non-pairwise coupling} that has recently been investigated in for coupled oscillators (see, for example,~\cite{Tanaka2011a, AshBicRod2016, Bick2017c}) and in a broader contexts such as ecological networks~\cite{Levine2017}.
Second, even if one assumes linearly weighted pairwise couplings, the {coupling} function itself is often assumed to be fairly simple in some sense:
In the case of weakly coupled oscillatory units, the interaction is typically 
{never assumed to vanish on any interval of phase differences.}

Here we analyse networks with more complex coupling functions that allow dynamical units to (effectively) dynamically decouple and recouple as the system evolves. More specifically, the dynamical systems we investigate have state-dependent interactions that arise through \emph{dead zones} in the coupling function. Intuitively speaking, in the dead zones of a {coupling} function there is no interaction. We call the complements \emph{live zones}; we formalize these concepts below. Although a wide range of adaptive and time-dependent networks have been studied in the past, this sort of coupling between network structure and dynamics has been overlooked in many contexts, probably because the coupling functions considered may be thought to be pathological (the functions have nontrivial variation at some regions and trivial variation at others and so cannot be analytic).
State-dependent dynamics induced by coupling with dead zones relate dynamical models in systems biology, through piecewise linear dynamics characterized by thresholds (see, for example, \cite{Glass1973,Glass-Pasternack1978a,Edwards2000,Hahnloser2003a,Huttinga2018,Curto2019}) or continuous modelling (e.g.,~\cite{Goodwin1965,Tyson}).
However, the investigation in these specific settings have focused almost exclusively on the asymptotic behaviour (whether synchronized, periodic~\cite{Poignard2018g, Poignard2016g, Hastings1977, Tyson3}, or chaotic~\cite{Poignard2013,Shahrear2015}) of the underlying high-dimensional coupled system rather than on the dynamics of the effective interactions between individual subsystems of a network. This is also the case in~\cite{Soule2003} and~\cite{Pecou2012}, where the authors have related the existence of circuits in the graph of interactions to the existence of multistable or stable limit cycles in phase space.

In this paper, we elucidate the interplay between dynamics and \emph{effective} interactions in networks of coupled phase oscillators. We first define the notion of a dead zone for such systems, which leads to the definition of an effective coupling graph at a particular state. The network dynamics are determined by the effective interactions: The dynamics at a particular point are determined by the (state-dependent) effective interactions which may change over time. Hence, the dynamics are determined by both the structural properties (what effective interactions are possible) and the dynamical properties (whether and how do the effective interactions change over time) of the system. Our contribution is threefold. First, we consider the structural question what effective interaction graphs are possible for a given network structure by careful design of coupling functions and examination of the dead zones. Second, we give a result on how the effective interactions do (or do not) change as time evolves. Third, we give instructive examples how dead zones shape the network dynamics; in particular, for a fully connected network and a coupling function for which the dynamics are fully understood, a single dead zone can induce nontrivial periodic dynamics.

While we restrict ourselves to dead zones in coupled phase oscillators, we note that these concepts are likely to be applicable to more general network dynamical systems in a wider range of contexts.
Indeed, even in networks that appear structurally simple (for example, networks that are all-to-all coupled from a structural point of view) the existence of dead zones can induce new dynamics. A decomposition of phase space into regions of identical effective interactions yields a natural coarse graining of the system: It can be understood in terms of the transitions between effective interactions similar to the state transition diagrams in~\cite{Glass-Pasternack1978a}.
Such a dynamical decomposition can provide a framework for network dynamical systems with coupling that has ``approximate'' dead zones---regions where the coupling is small but nonzero.
Hence, we anticipate that notable examples and applications may arise in systems biology as discussed above, where many studies use this type of active/inactive interactions, neuroscience, where for example state- and time-dependent interactions may arise for example through mechanisms such as spike-time dependent plasticity~\cite{refSTDP} or refractory periods, or continuous opinion models~\cite{Deffuant2000,Hegselmann2002} where agents only interact if their opinion is sufficiently close.

\subsection{Dead zones for phase oscillators}
\label{sec:DeadzonesIntro}

The particular class of system that we study in this paper has particularly simple unit dynamics (coupled phase oscillators) and pairwise coupling so that a network of interactions and pairwise coupling functions are appropriate. These models arise naturally in a range of applications where there are coupled limit cycle oscillators and the coupling is weak compared to the limit cycle stability~\cite{Izhikevich2007, AshComNic2016}. More precisely, we assume that the phase~$\theta_k\in \bbT:=\bbR/(2\pi\bbZ)$ of oscillator~$k\in\sset{1, \dotsc, N}$ evolves according to
\begin{equation}\label{eq:PhaseOscCoup}
\dot\theta_k=\omega + \sum_{j=1}^N A_{jk}\rg(\theta_j-\theta_k),
\end{equation}
where~$\omega$ is the fixed {intrinsic frequency} of all oscillators,
$A_{jk}\in\sset{0,1}$ encode the coupling topology between oscillators (we assume no self-coupling, $A_{kk}=0$), 
and the (non-constant) \emph{coupling function $\rg:\bbT\to\bbR$} determines how the oscillators influence each other.
A graph~$\bfA$ is associated with the adjacency matrix~$(A_{jk})$---the \emph{structural coupling graph} that encodes whether oscillator~$k$ can receive input from oscillator~$j$.
We constrain ourselves here to phase oscillator networks where the oscillators have the same intrinsic frequency and the coupling function~$\rg$ is the same between all pairs of oscillators.

The coupling topology relates to dynamical properties of network dynamical systems such as~\eqref{eq:PhaseOscCoup}. For commonly studied coupling functions, properties of the structural coupling graph~$\bfA$ such as its spectrum~\cite{Fiedler1973, agaev2000} determine, for example, synchronization properties (complete or partial) of the network~\eqref{eq:PhaseOscCoup}: This is used in the master stability function approach of Pecora and Caroll~\cite{Pecora1998, Barahona2002} and the work of Wu and Chua~\cite{Wu1996}, who revealed the role played by the spectral gap and the spectral radius of~$\bfA$. Various conditions for complete synchronization of networks with arbitrary graph structures have been found using spectral properties~\cite{Belykh2005, Nishikawa2006a, Motter2003, Changpin2007, Pereira2014}.
Consequently, this helps understand the effects of structural perturbations on the synchronizability of networks, including~\eqref{eq:PhaseOscCoup} and networks where the unit dynamics are more complex (see~\cite{Pogromsky2001, Milanese2010, Nishikawa2017, Poignard2018, Poignard2019}).

Note that it is not necessarily sufficient to consider the structural coupling graph~$\bfA$ to determine dynamical properties: This is particularly the case if the coupling function~$\rg$ has~\emph{dead zones}, i.e., if it is zero over some interval of phase differences. In the presence of dead zones, we will define an \emph{effective coupling graph} of~\eqref{eq:PhaseOscCoup} as a subgraph of~$\bfA$, which encodes the effective interactions between oscillators at a particular point in phase space. By definition, the effective coupling graph is state-dependent and may change dynamically with time. As the system evolves, the network may even decouple into several components under the influence of dead zones. These networks with time-
and state-dependent links can also be viewed in the framework of ``asynchronous networks''~\cite{Bick2015, Bick2017a}.

Even though we assume that the uncoupled units are very simple and the functional form of interactions is the same, the possible dynamics of~\eqref{eq:PhaseOscCoup} may be very complex \cite{AshComNic2016}. We therefore mostly restrict the equations~\eqref{eq:PhaseOscCoup} to the case where the coupling is all-to-all (and thus fully symmetric), that is, $A_{kj}=1$ for all $j\neq k$, and the phase $\theta_k\in\bbT$ evolves according to
\begin{equation}\label{eq:PhaseOscSN}
\dot\theta_k=\omega + \sum_{j=1,j\neq k}^N \rg(\theta_j-\theta_k)
\end{equation}
for $k=1,\dotsc, N$. In spite of the high degree of symmetry of the system, the system shows a very rich variety of behaviour that includes synchronization~\cite{Acebron2005}, clustering~\cite{OroMoeAsh2009}, heteroclinic dynamics~\cite{Ashwin2007}, and chaos~\cite{Bick2011}; see also~\cite{Ashwin1992,OroMoeAsh2009,Ashwin2016} for a discussion of the dynamics and bifurcations of~\eqref{eq:PhaseOscSN} and 
see~\cite{Pikovsky2015a} for a recent review.

Looking at the structural coupling graph the network associated with~\eqref{eq:PhaseOscSN} is rather simple, since the network is \emph{fully connected}.
However, this also means that there is a rich set of $2^{N(N-1)}$ subgraphs corresponding to setting (off-diagonal) entries of~$\bfA$ to zero. For a coupling function with dead zones, this means there is a very rich set of effective coupling graphs that can occur.

Our goal here is {to} explore some connections between properties of such dead zones, the effective coupling and the typical dynamics for such networks with these coupling functions for networks of the form~\eqref{eq:PhaseOscCoup} or \eqref{eq:PhaseOscSN}. As an example, Figure~\ref{fig:N=5schematic} shows some possible effective couplings that can be achieved by~\eqref{eq:PhaseOscSN} with $N=5$ and choices of coupling function with dead zones.
Section~\ref{sec:setting} presents a setting in which these dead zones can be defined. It also presents conditions in Proposition~\ref{prop:LocalProdStruct} for local skew product structure that appears in the dynamics due to the dead zones. We then address the following questions:
\begin{itemize}
\item[\bf Q0:] Given any subgraph of the structural coupling graph, is there a coupling function such that this subgraph is {\em realised} as the effective coupling graph for some point in the phase space?
\item[\bf Q1:] What is the relation between the coupling function, the set of possible subgraphs that can be realised, and the points where these realisations happen?
\item[\bf Q2:] How do the dynamics and effective couplings influence each other?
\end{itemize}

\begin{figure}
\centerline{
\includegraphics[width=12cm]{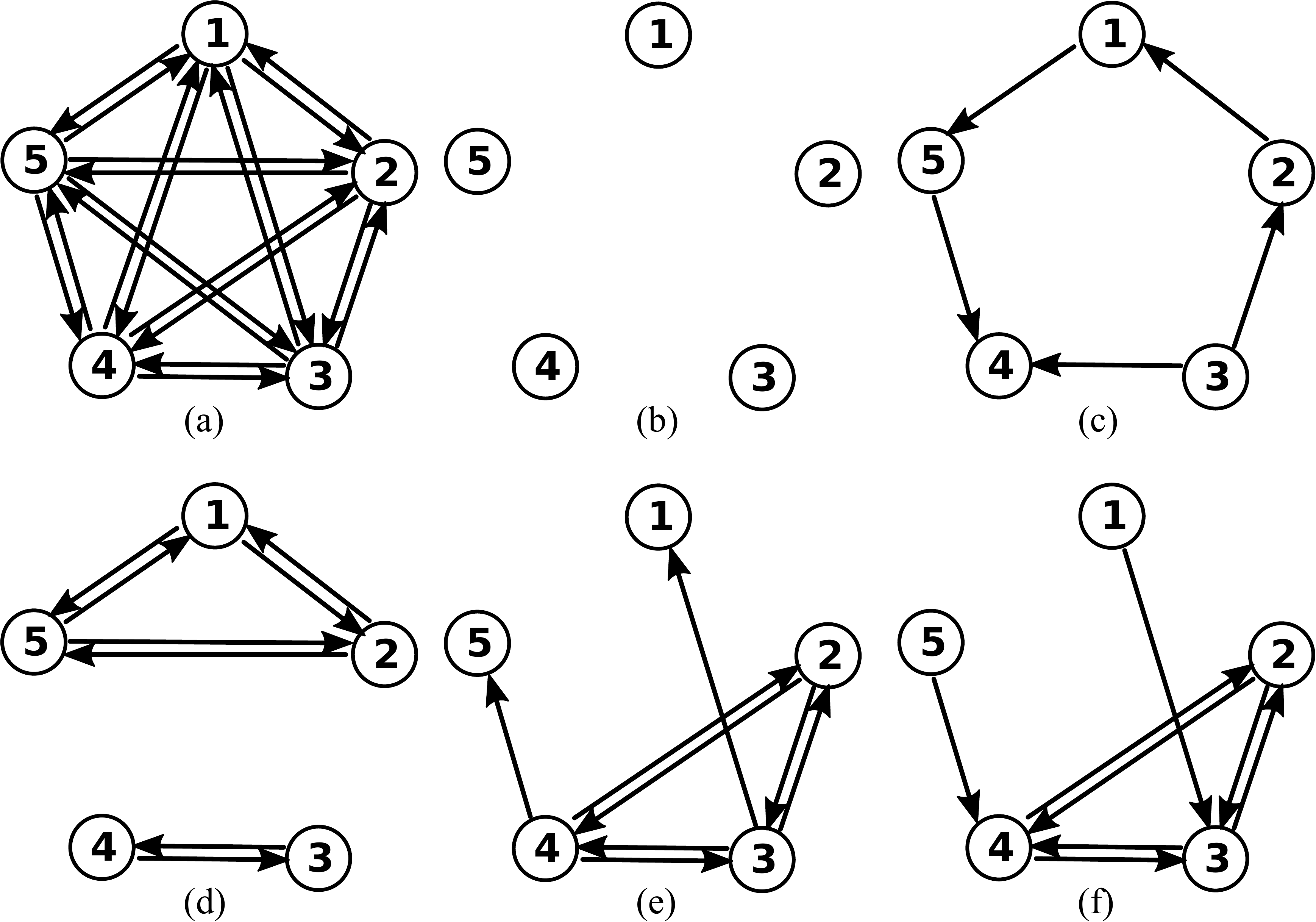}
}
\caption{(a) Coupling for the graph~$\bfK_5$ corresponding to the fully connected network~\eqref{eq:PhaseOscSN} with $N=5$. (b-f) show five examples of the $2^{5\times4}=1048576$ possible embedded subgraphs of~(a), i.e., having the same number of nodes as (a): by Proposition~\ref{prop:gentheta} we can show that all of these and more can be realised as effective coupling graphs for a coupling function~$\rg$ with dead zones. Panels~(b) and~(d) shows graphs with more than one component: (b) is the ``empty'' graph with no edges, (c) is a cycle of length~$5$ and (d,f) have nontrivial structure.  A typical trajectory of the system for such a system will visit several different effective coupling graphs under time-evolution. While (e) and (f) show similar structure Proposition~\ref{prop:realiseattractor}
shows that only~(e) can be realised in a dynamically stable manner as it contains a spanning diverging tree.}
\label{fig:N=5schematic}
\end{figure}

Section~\ref{sec:Realization} discusses some results related to these questions concerning effective coupling graphs: 
Proposition~\ref{prop:gentheta} answers {\bf Q0} positively in that it shows that any directed graph can be realised as an effective coupling for a suitable coupling function. A refinement of this proposition is given by Proposition~\ref{prop:delta} by showing it is possible to do it at a given length of live zone (small enough with respect to the number of nodes). We also consider specific subcases of {\bf Q1} by showing how symmetries of the points in the torus determine partially the effective coupling graphs.
Section~\ref{sec:transitions} moves on to {\bf Q2} and consider the interaction between effective coupling and dynamics: Proposition~\ref{prop:realiseattractor} (and Corollary~\ref{corstable}) show that any graph that admits a spanning diverging tree can not only be realised at some point in phase space but also this can be made dynamically stable. More generally, it seems that the interaction of dynamics and dead zones may be quite complex and so we explore some examples. Section~\ref{sec:examples} looks in detail at the realised effective coupling graphs for $N=2$ and $N=3$ with one dead zone. Finally, Section~\ref{sec:discuss} explores uses, generalisations and applications of these ideas.

\subsection{Graph theoretic preliminaries}

We briefly introduce some notions (and notation) that are used in this paper (see, e.g.,~\cite{Diestel2017} for more background on graph theory). Recall that a \emph{directed graph}~$\bfG$ is a pair $\bfG=(V,E)$ with a finite set of {\em vertices}~$V$ and directed {\em edges} $E\subset V^2$ between vertices. Depending on the context, we write~$V(\bfG)$, $E(\bfG)$ to denote the vertices and edges of~$\bfG$. A pair $(j, k)\in E$ is an \emph{edge from vertex~$j$ to vertex~$k$}. Since the graphs we consider here relate to network dynamics, we use terms {\em vertex/node}, and {\em edge/link} interchangeably. We will assume that the graphs do not contain self-loops, i.e., $(k,k)\not\in E$ for any~$k$. A vertex $k\in V$ is said to have an incoming edge if there exists another vertex~$j$ such that $(j,k) \in E$. Any graph~$\bfG$ can be identified with an \emph{adjacency matrix~$A^{\bfG}$} with coefficients $A^{\bfG}_{kj}=1$ if $(k,j)\in E(\bfG)$ and $A^{\bfG}_{kj}=0$ otherwise. We say the graph~$\bfG$ is \emph{undirected} if~$A^{\bfG}$ is symmetric, i.e., $A^{\bfG}$~is equal to its transpose, or equivalently, $(j,k)\in E$ if and only if $(k,j)\in E$. This means that an undirected graph has an even number of directed edges (see Figure~\ref{fig:N=5schematic}). Finally, a graph $\bfH=(V',E')$ is a \emph{subgraph of~$\bfG$}, and we write $\bfH\subset\bfG$, if $V'\subset V$ and $E'\subset E$. In this paper it will be convenient to consider a subgraph~$(V',E')\subset\bfG$ as an \emph{embedded subgraph}~$(V,E')$ by including all vertices~$V$ of $\bfG$.

Write $V_N=\sset{1,\dotsc,N}$. If~$\bfG=(V_N,E)$ and $\bfH\subset\bfG$ is an embedded subgraph, then the associated adjacency matrices~$A^{\bfG}$, $A^{\bfH}$ are $N\times N$ matrices. The \emph{fully connected graph}~$\bfK_N$ is the graph on~$V_N$ with $(j,k)\in E(\bfK_N)$ for all $j\neq k$; it has $N(N-1)$ edges. Similarly, let~$\bes_N = (V_N, \emptyset)$ denote the \emph{empty graph} with no edges. Note that~$\bfK_N$ and~$\bes_N$ are undirected {(see Figure~\ref{fig:N=5schematic})}. For $(p_1, p_2, \dotsc, p_r)\in V_N^r$,
let~$\bfP_{p_1,\dotsc,p_r}$ denote the \emph{directed path} with vertices ${p_1,\dotsc,p_r}$, i.e., the subgraph of~$\bfK_N$ with edge set
\[E(\bfP_{p_1,\dotsc,p_r}) = \set{(p_q, p_{q+1})}{q=1, \dotsc, r-1}\]
and similarly,
let $\bfC_{p_1,\dotsc,p_r}$ be the \emph{directed cycle} with vertices ${p_1,\dotsc,p_r}$ and edges $E(\bfC_{p_1,\dotsc,p_r}) = E(\bfP_{p_1,\dotsc,p_r})\cup\sset{(p_r, p_1)}$. The \emph{undirected path}~$\pugraph_{p_1,\dotsc,p_r}$ and \emph{undirected cycle}~$\cugraph_{p_1,\dotsc,p_r}$ are obtained by adding the reverse edges to~$\bfP_{p_1,\dotsc,p_r},\bfC_{p_1,\dotsc,p_r}$, respectively. Finally, let $\bfK_{p_1,\dotsc,p_r}$ be the \emph{fully connected subgraph} {on the set of nodes $\sset{p_1,\dotsc,p_r}$.
When convenient, we will identify the graphs $\bfP_{p_1,\dotsc,p_r}$, $\bfC_{p_1,\dotsc,p_r}$, and~$\bfK_{p_1,\dotsc,p_r}$ (and their undirected versions) with their corresponding embedded subgraphs with vertices~$V_N$.

A directed graph $\bfG=(V,E)$ is \textit{strongly connected} if given any two vertices in~$V$ there exists a directed path of edges in~$E$ between these two nodes. A directed graph~$\bfG$ is said to be \textit{weakly connected} if it is not strongly connected and its underlying undirected graph, obtained by ignoring the orientations of the edges, is strongly connected. A \textit{spanning diverging tree of $\bfG=(V,E)$} is a weakly connected subgraph of~$\bfG$ such that one node (the root node) has no incoming edges and all other nodes have one incoming edge {(for instance, the graph in Figure~\ref{fig:N=5schematic}(e) contains such a tree)}. Lastly, for any graph $\bfG=(V,E)$, an \textit{independent set $S$} is a set of nodes included in~$V$, for which any two nodes of~$S$ are never connected by an edge in~$E$. We say that~$\bfG$ is a $k$-partite graph if~$V$ admits a partition in~$k$ distinct independent sets.

\subsection{Symmetries, dynamics and graphs}

Let~$\bfG$ be a graph with $V(\bfG)=V_N$ and let~$\bbS_N$ be the symmetric group of all permutations of $V_N=\sset{1,\dotsc, N}$. The \emph{automorphisms of~$\bfG$}, denoted by
\[
\G(\bfG)=\set{\g\in \bbS_N}{A_{\g(k)\g(j)}=A_{jk} \text{ for all }j,k\in V_N},
\]
form a subgroup of~$\bbS_N$ under composition. Define the set of embedded subgraphs
\begin{align*}
\cH(\bfG)&=\set{\bfH=(V_N,E')}{\bfH\subset\bfG}
\end{align*}
and write $\cH_N=\cH(\bfK_N)$. Note that the group~{$\G(\bfG)$} naturally acts on~$\cH(\bfG)$: For $\bfH\in\cH(\bfG)$ and {$\g\in\G(\bfG)$} the image $\g \bfH$ is the graph with vertices~$V_N$ and edges
\[
E(\gamma \bfH) = \set{(\gamma (j), \gamma (k))}{(j,k)\in E(\bfH)}
\]
for {$\g\in\G(\bfG)$}. For this action, the \emph{isotropy group of the graph~$\bfH\subset \bfG$} is
\[
\Sigma_{\bfH} = \set{\g\in{\G(\bfG)}}{\g \bfH = \bfH}.
\]
Note that the isotropy group does not uniquely identify the subgraph, for example one can reverse the edges and get the same isotropy; however, it is a useful characterisation of the graph.

The group~$\bbS_N$ acts on~$\bbT^N$ by permuting components. 
Let~$\bfG\in\cH_N$. 
For $\Sigma\subset\G(\bfG) \subset\bbS_N$} we define the \emph{fixed point space} $\fix(\Sigma)=\set{\theta\in\bbT^N}{\g(\theta)=\theta\text{ for all }\g\in\Sigma}$. For a given $\theta\in\bbT^N$, the \emph{isotropy subgroup} of~$\theta$ is the group action is
$\Sigma_{\theta}=\set{\g\in {\G(\bfG)}}{\g(\theta)=\theta}$. 
The symmetries have a number of dynamical consequences for the oscillator network~\eqref{eq:PhaseOscCoup} for~$\bfG$ above being the structural coupling graph~$\bfA$: Note that~\eqref{eq:PhaseOscCoup} is equivariant with respect to the action of~{$\Gamma(\bfA)\times \bbT$ where $\Gamma(\bfA)\subset \bbS_N$} acts via permutation of the oscillators and $\phi\in\bbT$ acts via phase shifts
\begin{equation}
(\theta_1,\ldots\theta_N)\mapsto(\theta_1+\phi,\ldots\theta_N+\phi).
\end{equation}
The fixed point space of any isotropy subgroup of {$\Gamma(\bfA)\times \bbT$} is dynamically invariant \cite{Ashwin1992}. It is often useful to consider behaviour of \eqref{eq:PhaseOscCoup} in terms of the group orbits of~$\bbT$. Equivalently, the quotient by $\bbT$ corresponds to considering the dynamics in phase difference coordinates, and relative equilibria (equilibria for the quotient system) typically correspond to periodic orbits for the original system. 

For the structural coupling graph~$\bfA=\bfK_N$, we obtain the all-to-all coupled oscillator network~\eqref{eq:PhaseOscSN}, which is $\Gamma(\bfK_N)\times \bbT=\bbS_N\times \bbT$ equivariant. In this case, the dynamics on the full phase space~$\bbT^N$ are completely determined by the dynamics on the \emph{canonical invariant region (CIR)} \cite{Ashwin1992,Ashwin2016}
\begin{equation}
\cC = \set{\theta = (\theta_1, \dotsc, \theta_N)}{\theta_1<\theta_2<\dotsb<\theta_N<2\pi}.
\end{equation}
The full synchrony and splay phase configurations
\begin{align*} 
\Sync &= (\phi, \dotsc, \phi),&
\Splay &= \left(\phi, \phi+\frac{2\pi}{N}, \dotsc, \phi+\frac{(N-1)2\pi}{N}\right)\in\cC
\end{align*}
are relative equilibria of the dynamics. There is a residual action of~$\Z_N := \Z/N\Z$ on the canonical invariant region and~$\Splay$ is the fixed point of this action~\cite{Ashwin1992}.

\section{From dead zones to effective coupling graphs}
\label{sec:setting}

In this section we define dead zones for a {coupling} function and introduce the resulting effective coupling graph and its properties. We will restrict to a suitable class of coupling functions that have dead zones but are otherwise smooth and in general position; clearly this could be easily generalised for example to coupling functions with only finite differentiability.

\begin{defn}
Suppose that $\rg:\bbT\to\bbR$ is a smooth~$2\pi$-periodic function.
\begin{itemize}
\item
A coupling function~$\rg$ is \emph{locally constant at $\theta^0\in\bbT$ with value $c\in\bbR$} if there is an open set~$U$ with $\theta_0\in U\subset\bbT$ such that $\rg(U)\equiv c$. Define~$\CZ(\rg)$ to be the set of locally constant points of~$\rg$.
\item
A coupling function~$\rg$ is \emph{locally null at $\theta^0\in\bbT$} if it is locally constant with $c=0$. Let~$\DZ(\rg)\subset \CZ(\rg)$ denote the set of locally null points of~$\rg$.
\item
A coupling function~$\rg$ has \emph{simple dead zones} if~$\DZ(\rg)$ has finitely many connected components and $\CZ(\rg)=\DZ(\rg)$, i.e., if there is a finite set of locally constant regions, and all are locally null.
\end{itemize}
\end{defn}

\begin{defn} 
Let~$\rg$ be a coupling function with simple dead zones. Any connected component of~$\DZ(\rg)$ is a \emph{dead zone} of~$\rg$. Connected components of the complements $\LZ(\rg)=\bbT\setminus \DZ(\rg)$ are \emph{interaction} or \emph{live zones}.
\end{defn}

Here, we will only consider the case of \emph{simple dead zones}:
in the rest of the paper, we will implicitly consider only \emph{coupling functions with simple dead zones}.
The class of coupling function with simple dead zones excludes (smooth approximation of) piecewise constant coupling functions. These may have nontrivial dynamics that are solely given by the different frequencies in the region where the coupling function is locally constant. Such nontrivial dynamics are clearly of interest in some applications, but is beyond the scope of this paper.

\begin{defn}
We say that~$\rg$ is \emph{dead zone symmetric} if $-\DZ(\rg)=\DZ(\rg)$ modulo~$2\pi$, i.e., if whenever $\phi\in\bbT$ is in a dead zone, then~$-\phi$ also belongs to a dead zone.
\end{defn}

{As an illustration, Figure~\ref{fig:dzN3f1} in Section \ref{sec:examples} provides examples of coupling functions (which are dead zone symmetric or not) with one dead zone.}

\subsection{Effective coupling graphs and their symmetries}

Suppose that~$\rg$ is a coupling function for~\eqref{eq:PhaseOscCoup} with structural coupling  graph~$\bfA$ given by the adjacency matrix~$(A_{jk})$ and let $\theta\in\bbT^N$. We say a node~$k$ is \emph{$\rg$-effectively influenced by node~$j$ at~$\theta$} for~\eqref{eq:PhaseOscCoup} if $A_{jk}=1$ and $\theta_j-\theta_k\not\in\DZ(\rg)$. 

\begin{defn}
The \emph{effective coupling graph}~{$\cG_{\rg,\bfA}(\theta)$} of~\eqref{eq:PhaseOscCoup} with coupling function~$\rg$ at~$\theta\in\bbT^N$ is the graph on~$N$ vertices with edges
\begin{align*}
E({\cG_{\rg,\bfA}(\theta)})=\set{(j,k)}{A_{jk}\neq 0 \text{ and }\theta_j-\theta_k\not \in \DZ(\rg)}.
\end{align*}
\end{defn}

Conversely, an edge $(j,k)\not\in E({\cG_{\rg,\bfA}(\theta)})$ if $A_{jk}=0$ (the edge is not contained in~$\bfA$) or $\theta_j-\theta_k\in\DZ(\rg)$ (the phase difference is in a dead zone).

Clearly ${\cG_{\rg,\bfA}(\theta)}\subset \bfA\subset \bfK_N$, and this will be a proper subgraph (that is, it differs from~$\bfA$ by at least one edge) for some $\theta\in\bbT^N$ if~$\rg$ has at least one dead zone. 
For the system~\eqref{eq:PhaseOscCoup} with coupling function~$\rg$ and given $\bfH\subset\bfK_N$, define
\begin{align}
{\Theta_{\rg,\bfA}(\bfH)}=\set{\theta\in\bbT^N}{{\cG_{\rg,\bfA}(\theta)}=\bfH}.
\end{align}

\begin{defn}
If~${\Theta_{\rg,\bfA}(\bfH)}$ is not empty, then~$\bfH$ is \emph{realised as an effective coupling graph for~\eqref{eq:PhaseOscCoup} with coupling function~$\rg$}. Moreover, a graph~$\bfH$ \emph{can be realised as an effective coupling graph for~\eqref{eq:PhaseOscCoup}} if there exists a coupling function~$\rg$ for which~${\Theta_{\rg,\bfA}(\bfH)}$ is not empty.
\end{defn}

For the special case $\bfA = \bfK_N$, that is, the oscillator network~\eqref{eq:PhaseOscSN}, we simply write~$\cG_{\rg}(\theta)$ for the effective coupling graph with edges
\[E(\cG_\rg(\theta))=\set{(j,k)}{\theta_j-\theta_k\not \in \DZ(\rg)}.\]
Similarly we write
\begin{align}
{\Theta_{\rg}(\bfH)}=\set{\theta\in\bbT^N}{{\cG_{\rg}(\theta)}=\bfH}
\end{align}
for the regions of phase space with a particular effective coupling graph. Note that the sets $\Theta_{\rg}(\bfH)$, $\bfH\in\cH_N$, partition the CIR~$\cC$.

For particular structural coupling graphs~$\bfA$ of~\eqref{eq:PhaseOscCoup} there are a large number of symmetries, i.e., the automorphism group~$\Gamma(\bfA)$ may be large \cite{Pecora1998}; it is maximal for $\bfA=\bfK_N$. At the same time,~$\Gamma(\bfA)$ acts on the underlying phase space. We now show how the symmetry of a point~$\theta\in\bbT^N$ relates to the symmetries of the effective coupling graph at~$\theta$.

\begin{lem}
Consider the system~\eqref{eq:PhaseOscCoup} with structural coupling graph~$\bfA$ and any coupling function~$\rg$. For any $\theta\in\bbT^N$, we have ${\cG_{\rg,\bfA}(\g \theta)} = \g {\cG_{\rg,\bfA}(\theta)}$ for all~$\g\in\G(\bfA)$.
\end{lem}

\begin{proof}
Note that $[\g\theta]_k=\theta_{\g(k)}$, where $[\,\cdot\,]_k$ refers to the $k$th component, and so $[\g\theta]_j-[\g\theta]_k\in \DZ(\rg)$ if and only if 
$\theta_{\g(j)}-\theta_{\g(k)}\in\DZ(\rg)$. These are the edges of~$\g{\cG_{\rg,\bfA}(\theta)}$.
\end{proof}

\begin{cor}
\label{cor:sigma}
Consider the system~\eqref{eq:PhaseOscCoup} with structural coupling graph~$\bfA$ and any coupling function~$\rg$. For any $\theta\in\bbT^N$, we have
\[\Sigma_{\theta}\subset \Sigma_{{\cG_{\rg,\bfA}(\theta)}}\subset\G(\bfA).\]
\end{cor}

\begin{proof}
To see this, note that if $\g\in\Sigma_{\theta}$ then $\g\theta=\theta$ and so
${\cG_{\rg,\bfA}(\theta)}={\cG_{\rg,\bfA}(\g \theta)}=\g{\cG_{\rg,\bfA}(\theta)}$
which implies that $\g\in\Sigma_{\cG_{\rg,\bfA}(\theta)}$. 
\end{proof}

Note that the reverse containment of Corollary~\ref{cor:sigma} does not necessarily hold, for example if there are no dead zones (i.e., if~$\DZ(\rg)$ is empty) then clearly $\Sigma_{\theta}=\G(\bfA)$ for all~$\theta$.

\begin{rem}
Note that while~$\CIR$ is a fundamental region for the dynamics of the all-to-all coupled network~\eqref{eq:PhaseOscSN}, the effective coupling graphs can differ between symmetric copies of~$\CIR$.
\end{rem}

\subsection{Local (skew-)product structure and asynchronous networks}
\label{sec:SkewProducts}

In this section we show that the effective coupling graph at a point captures essential dynamical information. In particular, we have that, locally around a generic point, the vector field factorizes into factors that correspond to (weakly) connected components of the effective coupling graph.

For $v = \sset{v_1<v_2<\dotsb<v_r}\subset V_N$ let $\pi_v:\bbT^N\to\bbT^r$ denote the projection of~$\bbT^N$ onto the coordinates in~$v$. We write $\bbT^v := \pi_v(\bbT^N)$ and $\theta_v = (\theta_{v_1}, \dotsc, \theta_{v_r})$ are the coordinates in $\bbT^v$. Suppose that $v^1, v^2\subset V_N$ partition~$V_N$, that is, $v^1\cap v^2 = \emptyset$, $v^1\cup v^2 = V_N$.  Write $r_k = \tabs{v^k}$ for the length of the vector~$v^k$ and we identify $\bbT^N\cong\bbT^{r_1}\times\bbT^{r_2}$ with elements $\theta = (\theta_{v^1}, \theta_{v^2})$ through the natural isomorphism that reorders coordinates appropriately.

\begin{defn}
Consider a general ODE on the $N$-torus,
\begin{equation}
\dot \theta = F(\theta),
\label{eq:odeF}
\end{equation} 
with $\theta\in\bbT^N$, $F:\bbT^N\rightarrow \bbT^N$ some smooth function, and $v = (v^1, v^2)$ a partition of $V_N$.
\begin{enumerate}[(a)]
\item The system has a \emph{local skew-product structure $v^1\to v^2$} at $\theta\in\bbT^N$ if there is an open neighborhood~$U$ of~$\theta$ and functions $F_{(1)},F_{(2)}$ such that $F(\theta) = \big(F_{(1)}(\theta_{v^1}), F_{(2)}(\theta_{v^1}, \theta_{v^2})\big)$ for all $\theta\in U$.
\item The system has a \emph{local product structure with respect to~$v$} at $\theta\in\bbT^N$ if there is an open neighborhood~$U$ of~$\theta$ and functions $F_{(1)}$, $F_{(2)}$ such that $F(\theta) = \big(F_{(1)}(\theta_{v^1}), F_{(2)}(\theta_{v^2})\big)$ for all $\theta\in U$.
\end{enumerate}
\end{defn}
The second statement is equivalent to~$F$ having a local skew product structure $v^1\to v^2$ and $v^2\to v^1$ at $\theta$.
For $A\subset\bbT^N$ let $\inte(A)$ denote the interior of~$A$.

\begin{lem}\label{lem:Generic}
Consider the dynamics~\eqref{eq:PhaseOscCoup} with a coupling function~$\rg$. Generically, $\theta\in\inte(\Theta_{\rg,\bfA}(\bfH))$ for some~$\bfH$.
\end{lem}

\begin{proof}
We have that
\[\bbT^N\sm\bigcup_{H\subset\bfG}\inte(\Theta_{\rg,\bfA}(\bfH)) = \set{\theta = (\theta_1, \dotsc, \theta_N)\in\bbT^N}{\exists k\neq j:\theta_k-\theta_j\in\partial\DZ(\rg)}\]
is a union of finitely many algebraic sets.
\end{proof}

Note that this does not imply that all~$\Theta_{\rg,\bfA}(\bfH)$ have nonempty interior. Consider for example the fully connected network~\eqref{eq:PhaseOscSN} with~$\bfA=\bfK_N$ and a coupling function~$\rg$ such that $0\in\partial\DZ(\rg)$. Then $\Theta_\rg(\bfK_N)$ has an isolated point.

\begin{lem}\label{lem:NoDependence}
Consider the dynamics of~\eqref{eq:PhaseOscCoup} written in the form~\eqref{eq:odeF} for a coupling function~$\rg$. Suppose that $\theta\in\inte(\Theta_{\rg,\bfA}(\bfH))$. Then $(j,k)\not\in E(\bfH)$ if and only if there exists a neighborhood~$U\subset\bbT^N$ of~$\theta$ such that $(\partial_{\theta_j} F_k)(\theta)=0$ for all $\theta\in U$.
\end{lem}

\begin{proof}
Write $\theta=(\theta_1, \dotsc, \theta_N)\in\inte(\Theta_{\rg,\bfA}(\bfH))$.
First, suppose that $(j,k)\not\in E(\bfH)$. Since~$\theta\in\inte(\Theta_{\rg,\bfA}(\bfH))$ there exists a neighborhood $U\subset\inte(\Theta_{\rg,\bfA}(\bfH))$ of~$\theta$. Now $(\partial_{\theta_j} F_k)(\theta) = A_{kj}g'(\theta_j-\theta_k)=0$ for all $\theta=(\theta_1, \dotsc, \theta_N)\in U$ since either $A_{kj}=0$ or $\theta_j-\theta_k\in\DZ(\rg)$. 
Conversely, suppose that there exists an~$U$ such that $\partial_{\theta_j} F_k = 0$ on~$U$. But $0=\partial_{\theta_j} F_k = A_{kj}g'(\theta_j-\theta_k)$ on~$U$ which implies $A_{kj}=0$ or $g'(\theta_j-\theta_k)=0$ (on~$U$). In either case we have $(j,k)\not\in E(\bfH)$ by definition.
\end{proof}

Let $\bfG=(V, E)$ be a graph. A partition $\sset{v^1,v^2}$ of~$V$ is a \emph{graph cut}. Write $E_{v^p\to v^q}(\bfG) = \set{(j,k)\in E(\bfG)}{j\in v^p, k\in v^q}$ for the edges from vertices in~$v^p$ to vertices in~$v^q$. The \emph{cut-set} of $\sset{v^1,v^2}$ is $E_{v^1\to v^2}(\bfG)\cup E_{v^2\to v^1}(\bfG)$. The graph cut is \emph{directed $v^p\to v^q$} if $E_{v^p\to v^q}(\bfG)\neq\emptyset$, $E_{v^p\to v^q}(\bfG)=\emptyset$. The partition is \emph{disconnected} if $E_{v^p\to v^q}(\bfG)=E_{v^p\to v^q}(\bfG)=\emptyset$. The following result relates properties of the effective coupling graph at a given point~$\theta$ with the local properties of the dynamical system~\eqref{eq:PhaseOscCoup}.

\begin{prop}\label{prop:LocalProdStruct}
Consider a generic point~$\theta\in\bbT^N$.
\begin{enumerate}
\item There is a directed graph cut $v^1\to v^2$ for the effective coupling graph~$\cG_{\rg,\bfA}(\theta)$ iff the system~\eqref{eq:PhaseOscCoup} has a local skew-product structure $v^1\to v^2$ at~$\theta$.
\item The partition $\sset{v^1, v^2}$ of the effective coupling graph~$\cG_{\rg,\bfA}(\theta)$ is disconnected iff the system~\eqref{eq:PhaseOscCoup} has a local product structure $v^1\to v^2$ at~$\theta$.
\end{enumerate}
\end{prop}

\begin{proof}
Write $\bfH = \cG_{\rg,\bfA}(\theta)$. By Lemma~\ref{lem:Generic} we may assume $\theta\in\inte(\bfH)$. Suppose that~$\sset{v^1, v^2}$ is a partition. The assertion now follows from applying Lemma~\ref{lem:NoDependence} for any edge in $E_{v^1\to v^2}(\bfH)$ and $E_{v^2\to v^1}(\bfH)$, respectively.
\end{proof}

Recall that two nodes $v,w$ in a directed graph are weakly connected if there is a path of edges (irrespective of their direction) between them. A weakly connected component is a maximal weakly connected subgraph. The following is a direct consequence of Proposition~\ref{prop:LocalProdStruct}.

\begin{cor}
Consider a generic point~$\theta\in\bbT^N$ and let $\bfH=\cG_\rg(\theta)$. If $v^1, \dotsc, v^\ell$ is the partition of the vertices corresponding to the weakly connected components of~$\bfH$ then there is a neighborhood~$U$ of~$\theta$ such that~\eqref{eq:PhaseOscCoup} can be written in the form~\eqref{eq:odeF} with
\[F(\theta) = \big(F_{v^1}(\theta_{v^1}),\dotsc, F_{v^\ell}(\theta_{v^\ell})\big)\]
on $U\subset\bbT^{v^1}\times\dotsb\times\bbT^{v^\ell}\cong\bbT^N$.
\end{cor}

\begin{rem}
A network of coupled oscillators naturally defines a nontrivial asynchronous network as defined in~\cite{Bick2015}. Events occur when the effective coupling graph changes along a trajectory, which defines the ``event map''~$\mathcal{E}(\theta) = \cG_{\rg,\bfA}(\theta)$. Now the coupled oscillator network~\eqref{eq:PhaseOscCoup} written as~\eqref{eq:odeF} is determined by the state-dependent ``network vector field''~$F(\theta) = F^{\mathcal{E}(\theta)}(\theta)$. Moreover, the network structure is ``additive'' in the sense that the dynamics of each oscillator is determined by a sum of the contributions from other oscillators and the condition whether~$(j,k)\in\cG_{\rg,\bfA}(\theta)$ only depends on~$(\theta_k, \theta_j)$. In the language of~\cite{Bick2015, Bick2017a}, this means that the asynchronous network is ``functionally decomposable'' and ``structurally decomposable''. Finally, note that Proposition~\ref{prop:LocalProdStruct} implies that, generically, we have a local product structure, a condition for the spatiotemporal decomposition of the Factorization of Dynamics Theorem~\cite{Bick2017a}.
\end{rem}

\section{Realizing effective coupling graphs}
\label{sec:Realization}

In this section we aim to relate dead zones and effective coupling graphs, noting that the effective coupling graph $\cG_{\rg}(\theta)$ depends both on coupling function $\rg$ and choice of $\theta$. Recall that {\bf Q0} and {\bf Q1} in Section~\ref{sec:DeadzonesIntro} concern the set of effective coupling graphs that can be realised and
the goal of this section is to tackle such \emph{structural} problems. Recall that if $\Theta_{\rg,\bfA}(\bfH)\neq\emptyset$ then we say that the effective coupling graph~$\bfH$ is \emph{realised} for~$\rg$. Proposition~\ref{prop:gentheta} answers {\bf Q0} in the positive: for typical choice of $\theta$ (with trivial isotropy), all effective coupling graphs can be realised.
We consider two special cases of {\bf Q1}:
\begin{itemize}
\item[\bf Q1a] Given a point $\theta\in\bbT^N$, a structural coupling graph $\bfA$ and a graph~$\bfH\in\cH(\bfA)$, is there a coupling function~$\rg$ 
{for \eqref{eq:PhaseOscCoup} (resp. \eqref{eq:PhaseOscSN}})
such that $\mathcal{G}_{\rg,\bfA}(\theta)=\bfH$?
{(resp. $\mathcal{G}_{\rg}(\theta)=\bfH$)}? 
\item[\bf Q1b] Is there a coupling function that realises all possible effective coupling graphs? Specifically, for \eqref{eq:PhaseOscSN}, is there a coupling function~$\rg$ such that\footnote{Writing the image of a subset $A$ of the range of a function $F$ as $F(A):=\set{F(a)}{a\in A}$.} $\cG_{\rg}(\bbT^N)=\cH_N$?
\end{itemize}
Proposition~\ref{prop:patho} gives some constraints on answers of {\bf Q1a}. We do not have a complete answer to {\bf Q1a}, while Corollary~\ref{cor:allgraphs} answers {\bf Q1b} positively.  We also consider what possible effective coupling graphs will be realised for a coupling function~$\rg$: this is important if we wish to understand the dynamics of~\eqref{eq:PhaseOscSN} with a fixed coupling function. To some extent this is simply a computational question: For any~$\theta = (\theta_1, \dotsc, \theta_N)$ one has to determine which phase differences $\theta_j-\theta_k$ lie in a dead zone. 
We give some general results in this direction in Section~\ref{sec:couplingongraph}. We believe that typical coupling functions will not be able to realise more than a small subset of effective coupling graphs. Henceforth we mainly restrict to discussion of the fully connected network~\eqref{eq:PhaseOscSN} although several of the results easily generalize to~\eqref{eq:PhaseOscCoup}.

\subsection{Restrictions on the effective coupling graph imposed by $\theta$}

Here we tackle the questions $\bf  Q1a, \bf Q1b$ above by putting the emphasis on the point~$\theta$. Specifically, given~$\theta\in\bbT^N$, what do the properties of~$\theta$ impose on the effective coupling graphs of \eqref{eq:PhaseOscSN}? The isotropy of~$\theta\in\bbT^N$ has some important consequences on the possible effective coupling graphs realised at~$\theta$:

\begin{prop}\label{prop:patho}
Consider the all-to-all coupled oscillator network~\eqref{eq:PhaseOscSN} with coupling function~$\rg$. Let $\theta\in\cC\subset \bbT^N$ be fixed.
\begin{enumerate}[(i)]
\item
If~$\theta$ has isotropy~$\Sigma_{\theta}$ then~$\cG_{\rg}(\theta)$ must have at least the same isotropy.
\item
For full synchrony $\Sync = (a, \dotsc, a)$ we have $\cG_\rg(\Sync)\in\sset{\egraph_N, \bfK_N}$.
\item
Suppose there exists $0<a<2\pi/N$ such that $\theta_{k+1}-\theta_k=a$ for any $k\in \sset{1,\dotsc,N-1}$. Then one of the following cases occurs:
\begin{enumerate}[(1)]
\item The directed path $\mathbf{P}_{N,N-1,\dotsc,1}$
is a subgraph of $\cG_{\rg}(\theta)$ but $\mathbf{P}_{1,2,\dotsc,N}$ is not.
\item The directed path $\mathbf{P}_{1,2,\dotsc,N}$ is a subgraph of $\cG_{\rg}(\theta)$ but $\mathbf{P}_{N,N-1,\dotsc,1}$ is not.
\item The undirected path $\pugraph_{1,2,\dotsc,N}$ is a subgraph of $\cG_{\rg}(\theta)$.
\item $\cG_{\rg}(\theta)$ is a $n$-partite graph (with $n=[N/2]$ if $N$ is even or $n=[N/2]+1$ if not).
\end{enumerate}
\end{enumerate}
\end{prop}

\begin{proof}
(i)~Is an application of Corollary~\ref{cor:sigma}, notably $\theta$ with nontrivial isotropy will limit the possible networks to those that have at least the same isotropy.

(ii)~The claim follows from~(i): the effective coupling graph must have full symmetry and hence be either~$\bfK_N$ or~$\egraph_N$. 

(iii)~We have one of the following cases as (we illustrate case~(4) in Figure~\ref{Figpatho2}):
(1)~$a\in\LZ(\rg), 2\pi-a\in\DZ(\rg)$,
(2)~$a\in\DZ(\rg), 2\pi-a\in\LZ(\rg)$,
(3)~$a, 2\pi-a\in\LZ(\rg)$,
(4)~$a, 2\pi-a\in\DZ(\rg)$.
\begin{figure}
\centering
\includegraphics[width=8cm]{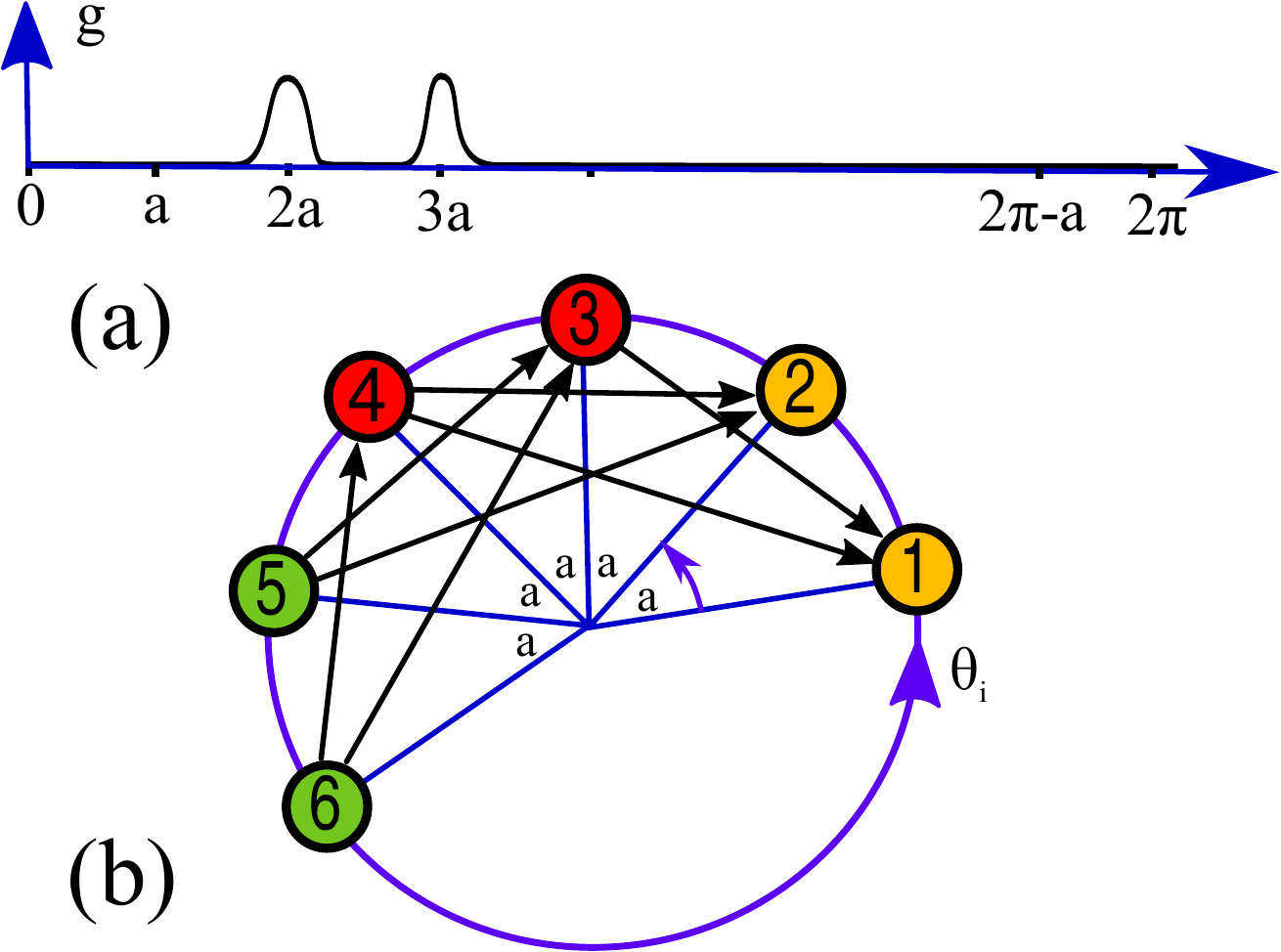}
\caption{
Illustration of the case (4) in Proposition~\ref{prop:patho}(iii) for $N=6$ oscillators: the coupling function~$\rg$ shown in Panel~(a) has two live zones centered at~$2a$ and~$3a$, the remainder consists of two dead zones. The diagram in Panel~(b) shows the phases~$\theta_k$ at one instant in time such that $\theta_j-\theta_i=a(j-i)$ for all $j>i$.
The effective coupling graph for the coupling function in~(a) is indicated by black arrows between the phases corresponding to individual nodes. This coupling graph is tripartite as indicated by the node colouring.}
\label{Figpatho2}
\end{figure}
Since all the differences $\theta_{i+1}-\theta_i$ are equal to~$a$ (and thus all $\theta_k-\theta_{k+1}$ equal to $2\pi-a$), then in case~(1) we have that $\bfP_{N,N-1,\dotsc,1}$ is a subgraph of~$\cG_{\rg}(\theta)$ but not~$\bfP_{1,2,\dotsc,N}$. Similarly for the cases~(2) and~(3).
In the case~(4) the vertices of~$\cG_{\rg}(\theta)$ can be partitioned into $n=[N/2]$ independent sets~$A_n$ if~$N$ is even (or into $n=[N/2]+1$ such sets if~$N$ is odd): 
namely the successive sets $A_1=\sset{1,2}$, $A_2=\sset{3,4}$, etc. This means that~$\cG_{\rg}(\theta)$ is a $n$-partite graph.
\end{proof}

While Proposition~\ref{prop:patho}(iii) limits to the splay configuration~$\Splay$ in a special case, the next Proposition gives a characterisation of effective coupling graphs that are realised for~$\Splay$.

\begin{prop}\label{prop:splay}
Consider the coupled oscillator network~\eqref{eq:PhaseOscSN} with coupling function~$\rg$.
\begin{enumerate}[(i)]
\item If $2\pi/N\in\LZ(\rg)$ then the directed cycle $\mathbf{C}_{1,2,\dotsc,N}$ is a subgraph of $\mathcal{G}_{\rg}(\Splay)$.
\label{prop:SplayOne}
\item Let $1<n<N$ and suppose that $2n\pi/N\in\LZ(\rg)$. Take indices modulo~$N$.
\label{prop:SplayTwo}
\begin{enumerate}
\item If $N=nm$ then for any $r\in\sset{1, \dotsc, m}$ the directed cycle $\bfC_{r,r+n,\dotsc,r+(m-1)n}$ is a subgraph of $\mathcal{G}_{\rg}(\Splay)$.
\item If $n$ does not divide~$N$, then the directed cycle $\bfC_{1,1+n,1+2n,\dotsc,1+N-n}$ is a subgraph of~$\mathcal{G}_{\rg}(\Splay)$.
\end{enumerate}
\end{enumerate}
\end{prop}

\begin{proof}
\eqref{prop:SplayOne} As in the proof of Proposition~\ref{prop:patho}(iii) we consider successive phase differences for~$\Splay = (\theta_1, \dotsc, \theta_N)$. We have $\theta_{k+1}-\theta_k = \theta_1-\theta_N = 2\pi/N$. Since $2\pi/N\in\LZ(\rg)$ by assumption, we have $(k, k+1)\in E(\mathcal{G}_{\rg}(\Splay))$ which proves that $\mathbf{C}_{1,2,\dotsc,N}\subset\mathcal{G}_{\rg}(\Splay)$.

\eqref{prop:SplayTwo} A similar argument proves the second assertion. Let $1<n<N$ and since $\theta_k-\theta_{k+n}\in\LZ(\rg)$ we have $(k,k+n)\in E(\mathcal{G}_{\rg}(\Splay))$. Now suppose that $N=nm+q$ with $0\leq q<m$. If $q=0$ we have $r+pn = r \mod N$ for $p=m<N$ which proves case~(a). If $q\neq 0$ then $r+pn = r \mod N$ only if $p\in N\Z$ which corresponds to case~(b).
\end{proof}

Recall that~$\Splay$ is the only fixed point of the residual action of~$\Z_N$ on the CIR~\cite{Ashwin1992}. The graph~$\mathcal{G}_{\rg}(\Splay)$ is invariant under the action of the symmetry by Proposition~\ref{prop:patho}(i). In cases~\eqref{prop:SplayOne} and~(\ref{prop:SplayTwo}b) of Proposition~\ref{prop:splay}, $\mathcal{G}_{\rg}(\Splay)$ contains one cycle involving all vertices, that is mapped to itself. By contrast, in case (\ref{prop:SplayTwo}a) of Proposition~\ref{prop:splay} there are~$m$ disjoint cycles that are permuted by the symmetry. The existence of cycles has also some immediate consequences for the connectedness of $\mathcal{G}_{\rg}(\Splay)$.

\begin{cor}
Consider system~\eqref{eq:PhaseOscSN} with coupling function~$\rg$ and suppose that the number of oscillators~$N$ is prime. Then either $\mathcal{G}_{\rg}(\Splay)=\egraph_N$ or $\mathcal{G}_{\rg}(\Splay)$ contains a directed cycle involving all~$N$ vertices, and so $\mathcal{G}_{\rg}(\Splay)$ is strongly connected.
\end{cor}

\begin{proof}
If $2n\pi/N\in\DZ(\rg)$ then $\mathcal{G}_{\rg}(\Splay)=\egraph_N$. Now suppose that there is an $1\leq n<N$ such that $2n\pi/N\in\LZ(\rg)$. Since~$N$ is prime, either Proposition~\ref{prop:splay}(i) and~(iib) applies. In either case, $\mathcal{G}_{\rg}(\Splay)$ contains a directed cycle involving all~$N$ vertices.
\end{proof}

The next result gives a positive answer to {\bf Q0}, providing that we avoid ``nongeneric'' choices of~$\theta$.

\begin{prop}\label{prop:gentheta}
For a generic choice of $\theta \in \bbT^N$, and for any subgraph $\bfH\in\cH_N$, there exists a coupling function~$\rg$ 
such that $\mathcal{G}_\rg(\theta)=\bfH$.
\end{prop}

\begin{proof}
Generically all the difference terms $\theta_j-\theta_k$ (when $j,k$ are ranging in $\sset{1,\dotsc,N}$) are distinct: therefore we can specify live zones that contain points $\theta_j-\theta_k$ if and only if the edge~$(j,k)$ is contained in~$\bfH$.
\end{proof}

Notice that the proof gives an upper bound on the number of dead and live zones needed to realise a given $\bfH$ as an effective coupling graph
(by choosing a coupling function and a point $\theta$): 
namely the bound given by the number of edges of $\bfH$. This bound is far from being optimal, notably for very ``regular" graphs:
for instance for any $\theta \in \bbT^N$, the graph $\bfK_N$ itself can be realised as a $\mathcal{G}_{\rg}(\theta)$ where $\rg$ has only one live zone and no dead zone. Corollary~\ref{cor:allgraphs} extends the same method of proof to show that one can, in principle, realise all subgraphs using one and the same $\rg$.

\begin{cor}
\label{cor:allgraphs}
There exists a coupling function~$\rg$ 
such that for any subgraph~$\bfH \in \cH_N$ there exists $\theta^0 = \theta^0(\bfH) \in \bbT^N$ such that $\cG_{\rg}(\theta^0)=\bfH$. 
\end{cor}

\begin{proof}
Enumerate all graphs $\bfH_n$ in~$\cH_N$ and choose a set $\sset{\theta^n = (\theta^n_1, \dotsc, \theta^n_N)\in\bbT^N}$ such that all phase differences $\theta^{n}_j-\theta^{m}_k$ are distinct. Now take a coupling function~$\rg$ such that for any~$n$ we have $\theta^{n}_j-\theta^{n}_k\in\LZ(\rg)$ if and only if $(j,k)$ is in~$E(\bfH_{n})$.
\end{proof}

A similar proviso holds here: such a constructed~$\rg$ will typically have a very large number of dead zones.

\subsection{Coupling functions for an interaction graph}
\label{sec:couplingongraph}

Given a coupling function~$\rg$, which properties of~$\rg$ imply certain effective coupling graphs realised by~$\rg$? On the other hand given $\theta \in \cC$, a structural coupling graph $\bfA$ and~$\bfH\in\cH(\bfA)$, how can one construct a coupling function~$\rg$ such that $\bfH=\mathcal{G}_{\rg,\bfA}(\theta)$? 
Among the different parameters characterising the coupling function $\rg$, the number of dead zones plays a major role in these questions, since it determines the shapes of the resulting effective coupling graphs. We thus make the following definition:

\begin{defn}
Let $n\in \mathbb{N}$. We denote by $\mathcal{F}(n)$ the set of coupling functions having $n$ dead zones\footnote{Note that if there are $n>1$ dead zones there must also be~$n$ live zones, while for $n=1$ there can be~$0$ or~$1$ live zones, and for $n=0$ there is necessarily one live zone.}.
\end{defn}

\begin{prop}
\label{prop:gprops}
Consider system \eqref{eq:PhaseOscSN} with coupling function $\rg$.
\begin{enumerate}[(i)]
\item The coupling function $\rg$ is dead zone symmetric if and only if all effective coupling graphs for $\rg$ are undirected.
\item Assume that $\rg\in\cF(1)$ is dead zone symmetric with $\LZ(\rg)=[-a, a]$. If $a<2\pi/N$, then for any $1\leq k\leq N$ and any sequence $k,\dotsc,k+p$ in $\sset{1,\dotsc,N}$ we have that~$\egraph_N, \bfK_N$ and the embeddings of $\pugraph_{k,\dotsc,k+p}$ and~$\bfK_{k,\dotsc,k+p}$ 
can be realised as effective coupling graphs for~$\rg$.
If $a =2\pi/N$, then $\bfK_N,\pugraph_{1,\dotsc,N},\cugraph_{1,\dotsc,N},$ and 
the embeddings of graphs $\pugraph_{k,\dotsc,k+p}$ and  $\bfK_{k,\dotsc,k+p}$
can be realised as effective coupling graphs for $\rg$.
\item Assume that $\rg\in\cF(1)$ is dead zone symmetric 
with $\LZ(\rg) = [\pi-a, \pi+a]$ and $a \leq 2\pi/N$.
Then $\egraph_N$ and $\bfK_N$ can be realised as effective coupling graphs for $\rg$.
\end{enumerate}
\end{prop}

\begin{proof}
(i) This item follows directly from the definition of a dead zone symmetric function.

(ii) Suppose that $a<2\pi/N$. Taking $\theta \in \cC$ such that all the successive differences $\theta_{i+1}-\theta_i$ are in $(a,2\pi/N)$, we have that $\mathcal{G}_{\rg}(\theta)=\egraph_N$: indeed any phase difference $\theta_j-\theta_k$, with $j>k$, will belong to the interval $(a,2\pi(N-1)/N)$ and therefore will be in~$\DZ(\rg)$. Similarly, taking~$\theta$ such that all the successive differences $\theta_{k+1}-\theta_k$ are strictly smaller than $a/(N-1)$ we have $\mathcal{G}_{\rg}(\theta)=\bfK_N$.
Now consider $1\leq k,p\leq N$ and a sequence $k,\dotsc,k+p$ in $\sset{1,\dotsc,N}$. Taking $\theta$ such that 
\begin{align*}
&\theta_{k+1}-\theta_k=\theta_{k+2}-\theta_{k+1}=\dotsb=\theta_{k+p}-\theta_{k+p-1}=a/2+\epsilon,\\
&\theta_{2}-\theta_1=\theta_{3}-\theta_{2}=\dotsb=\theta_{k}-\theta_{k-1}=a+\epsilon,\\
&\theta_{k+p+1}-\theta_{k+p}=\theta_{k+p+2}-\theta_{k+p+1}=\dotsb=\theta_{N}-\theta_{N-1}=a+\epsilon,
\end{align*}
where~$\epsilon$ is a sufficiently small positive real number, we have that $\mathcal{G}_{\rg}(\theta)=\pugraph_{k,\dotsc,k+p}$.
Similarly taking a point $\theta \in \cC$ such that
\begin{align*}
&\theta_{k+1}-\theta_k=\theta_{k+2}-\theta_{k+1}=\dotsb=\theta_{k+p}-\theta_{k+p-1}=\epsilon,\\
&\theta_{2}-\theta_1=\theta_{3}-\theta_{2}=\dotsb=\theta_{k}-\theta_{k-1}=a+\epsilon,\\
&\theta_{k+p+1}-\theta_{k+p}=\theta_{k+p+2}-\theta_{k+p+1}=\dotsb=\theta_{N}-\theta_{N-1}=a+\epsilon,
\end{align*}
with $\epsilon$ small enough, we have in this case that $\mathcal{G}_{\rg}(\theta)=\bfK_{k,\dotsc,k+p}$.
If $a=2\pi/N$, then the same reasoning applies.

(iii) This follows along similar lines as~(ii).
\end{proof}

Observe that, for the particular points $\theta=\Sync$ and $\theta=\Splay$ the coupling functions considered in Proposition~\ref{prop:patho}(i) and Proposition~\ref{prop:splay} can be taken in~$\mathcal{F}(0)$ and~$\mathcal{F}(1)$, respectively.
The proof of Proposition~\ref{prop:gentheta} 
constructs functions using many dead zones that might be very small. There are various questions one can pose about optimality. For example, given $\theta\in\cC$, a structural coupling graph $\bfA$ and a graph $\bfH\in\cH(\bfA)$, what is the minimum $k$ such that there is a $\rg\in\cF(k)$ such that $\bfH=\mathcal{G}_{\rg,\bfA}(\theta)$?
The proof of Proposition~\ref{prop:gentheta} gives an upper bound to this question {in the case of System \eqref{eq:PhaseOscSN}}, namely $n\leq 2^{\#E(\bfH)}$ (where $\#E(\bfH)$ is the number of edges of~$\bfH$), but does not give any information on the length of the live zones involved, which can possibly be arbitrarily small, and for which one may need to control the size. Proposition~\ref{prop:delta} below gives a lower bound, as a function of the number of nodes~$N$ (namely $\pi/2^{N-1}$), on the length of live zones $\delta>0$ for which it is possible
to realise any~$\bfH$ as an effective coupling graph (and this thanks to a coupling function of which live zones have length~$\delta$).

\begin{prop}\label{prop:delta}
Let $0<a<\pi/2^{N-1}$ and let $0<\delta< a$. Then, for any~$\theta\in\bbT^N$ with
\begin{align*}
\theta_{i+1}-\theta_i &\geq \theta_i-\theta_1+a, \qquad\text{for } i \in \sset{1,\dotsc, N-1},\\
\theta_N-\theta_1 &<\pi-a,
\end{align*}
and for any $\bfH\in \cH_N$, there exists and integer $n\leq \#E(\bfH)$ and a coupling function $\rg \in \mathcal{F}(n)$ 
with live zones of length at least~$\delta$ such that $\mathcal{G}_g(\theta)=\bfH$.
\end{prop}

\begin{proof}
Suppose that~$\theta$ satisfies the conditions required and~$\bfH$ is a subgraph of~$\bfG$. The idea is to construct a coupling function for which each live zone is precisely associated to only one edge of $\bfH$ (see Figure~\ref{Fig1}).

To do this we start with the coupling function identically equal to~$0$. Now, if $(2,1)\in E(\bfH)$ then we put a live zone of length $\delta$ centred at $\theta_2-\theta_1$. If not, then we set~$\rg$ to be locally null at $\theta_2-\theta_1$.
Next, we consider $\theta_3-\theta_2 \geq \theta_2-\theta_1+a$: If $(3,2)\in E(\bfH)$ we put a live zone of length~$\delta$ centered at $\theta_3-\theta_2$. This second live zone does not intersect the first one since $\theta_3-\theta_2> \theta_2-\theta_1+\delta$. If $(3,2) \notin E(\bfH)$ we let~$\rg$ be locally null at $\theta_3-\theta_2$. Then we have $\theta_3-\theta_1\geq \theta_3-\theta_2+a$ which satisfies $\theta_3-\theta_1 >\theta_3-\theta_2+\delta$; we can thus put a live zone of length~$\delta$ centered at $\theta_3-\theta_1$ if $(3,1)\in E(\bfH)$.

Repeating the process, we construct~$\rg$ by imposing the existence of a live zone of length~$\delta$ centered at any of the $\theta_i-\theta_k$ (with $1\leq k\leq i\leq N$) such that~$(i,k)$ is in~$E(\bfH)$. Since all these terms are smaller than~$\pi$, 
all the opposite values are determined and separated as well by a distance larger than~$a$. It is therefore possible to add live zones at the $\theta_k-\theta_i$ (with $1\leq k\leq i\leq N$) for which the edge $(k,i)$ is in~$E(\bfH)$.
\end{proof}

\begin{figure}
\centering
\includegraphics[width=12cm]{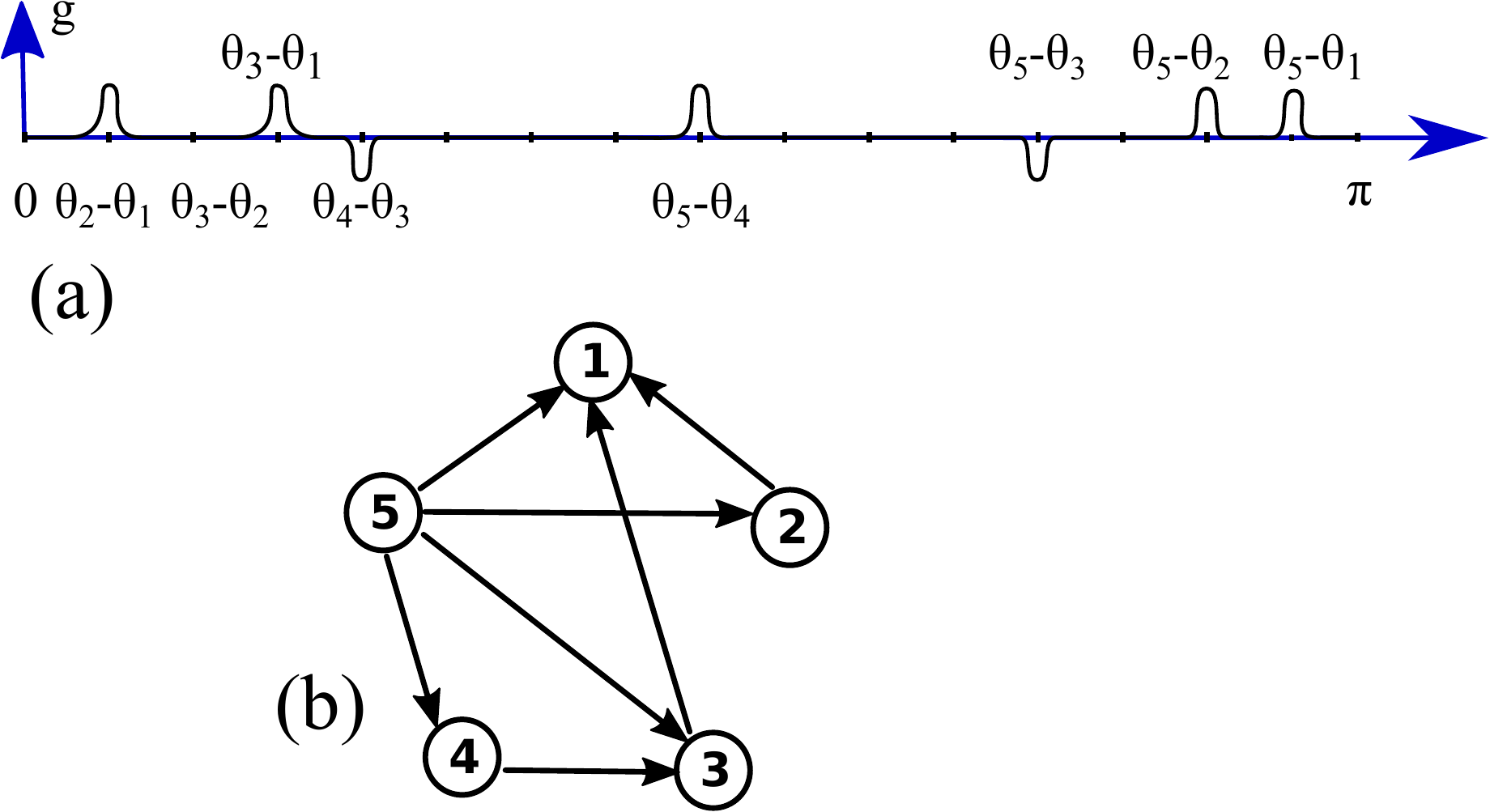}
\caption{
An example of directed graph (b) within $\bfK_5$ with $7$ edges realised as an effective coupling graph with a coupling function~$\rg$ in $\mathcal{F}(7)$ (a) constructed via the method described in the proof of Proposition~\ref{prop:delta}: in the bottom of the graph of $\rg$ we show the successive phase differences $\theta_{i+1}-\theta_i$ for which the edge $(i+1,i)$ is in $E(\bfH)$ (the ones for which $(i+1,i) \notin E(\bfH)$
are not shown), and at which we impose the existence of a live zone of $\rg$. Above the graph of $\rg$, we show the values of the phase differences that are determined by the values of the successive phase differences.}
\label{Fig1}
\end{figure}

\section{Dynamics of effective coupling graphs}
\label{sec:transitions}

The previous section considered the structural problem of understanding the effective coupling graph at some point in phase space. Now let~$\theta(t)=\varphi_t(\theta^o)$ be the solution of the phase oscillator network~\eqref{eq:PhaseOscCoup} with initial condition~$\theta^o$. Clearly, $\cG_{\rg,\bfA}(\varphi_t(\theta^o))$
defines an evolution on the set of effective coupling graphs. In this section, we briefly consider possible dynamics of these effective coupling graphs. 

Suppose that~$\bfH\in\cH_N$ is an effective coupling graph realised for~\eqref{eq:PhaseOscCoup} with coupling function~$\rg$ for some~$\theta$. 

\begin{defn}
The graph~$\bfH$ can be \emph{stably realised} if there is an asymptotically stable invariant open set~$B$ {such that} $B \subset \Theta_{\rg,\bfA}(\bfH)$. {Moreover,} if $B = \Theta_{\rg,\bfA}(\bfH)$, {then we say that the} graph~$\bfH$ is \emph{completely stably realised}.
\end{defn}

In other words, for a stably realised effective coupling graph~$\bfH$, there is {an} open set of~$\theta^o$ such that 
\begin{equation}
{\cG_{\rg,\bfA}(\varphi_t(\theta^o))=\bfH}
\end{equation}
for large enough~$t$. If~$\bfH$ is completely stably realised then this holds for all~{$\theta^o\in\Theta_{\rg,\bfA}(\bfH)$}.

By constructing a coupling function with a stable (relative) equilibrium, we now strengthen Proposition~\ref{prop:gentheta} to show that for any ``sufficiently connected''~$\bfH$ there exists a~$\rg$ such that the effective coupling graph~$\bfH$ can be stably realised. To prove this result, recall the following spectral graph property~\cite{agaev2000,agaev2009}:

\begin{prop}{\bf \cite[Corollary 1]{agaev2009}}\label{prop:agaev2009}
Let~$\bfH$ be a graph admitting a spanning diverging tree. Consider the Laplacian matrix~$L^{\bfH}$ with coefficients \begin{align*}
L^{\bfH}_{jk}=\begin{cases}
-A^{\bfH}_{jk}&\text{if }j \neq k,\\
\sum_{\ell=1,\ell \neq k}^NA^{\bfH}_{\ell k}&\text{if }k=j,
\end{cases}
\end{align*}
where~$A^{\bfH}$ denotes the adjacency matrix of the graph~$\bfH$.
Then the multiplicity of the eigenvalue~$0$ in the spectrum of~$L^{\bfH}$ is one.
\end{prop}

We use this to prove the following result:

\begin{prop}
\label{prop:realiseattractor}
For any $\bfH\in \cH_N$ admitting a spanning diverging tree, there is a coupling function~$\rg$ such that the oscillator network~\eqref{eq:PhaseOscSN} has a locally asymptotically stable relative equilibrium $(\Omega t +\theta^o_1,\dotsc ,\Omega t +\theta^o_N)$ satisfying $\cG_{\rg}(\theta^o)=\bfH$. In other words, there exists a coupling function~$\rg$ that stably realises~$\bfH$.
\end{prop}

\begin{proof}
Set $\mathbf{1} = (1, \dotsc, 1)\in\bbR^N$ and for this proof write $\Omega t \cdot \mathbf{1}+\theta^o$ the point $(\Omega t +\theta^o_1,\dotsc ,\Omega t +\theta^o_N)$. The proof follows in two steps.

\emph{Step 1.} First, take any $\bfH\in \cH_N$ and choose a generic point~$\theta^o$ such that all the terms $\theta^o_j-\theta^o_k$ (with $j\neq k$) are distinct. In this first step we are going to construct a coupling function $\rg$ with specific conditions written below, for which~$\Omega t \cdot \mathbf{1}+\theta^o$ is a relative equilibrium point of~\eqref{eq:PhaseOscSN}.

Following~\cite{OroMoeAsh2009}, we consider a coupling function~$\rg$ satisfying the conditions
\begin{align*}
\text{(i)}\quad &\cG_{\rg}(\theta^o)=\bfH,\\
\text{(ii)}\quad & \rg'(\theta^o_j-\theta^o_k)>0
\qquad\text{for all }(j,k) \in E(\bfH),\\
\text{(iii)}\quad &\frac{1}{N}\sum_{j=1}^N \rg(\theta^o_j-\theta^o_k)=\frac{1}{N}\sum_{j=1}^N \rg(\theta^o_j-\theta^o_1) \qquad\text{for }k = 1,\dotsc, N.
\end{align*}
Observe that such a coupling function~$\rg$ {satisfying (i),(ii),(iii)} exists, since by choice of~$\theta^o$, we can associate to any term $\theta^o_j-\theta^o_k \in \LZ(\rg)$ (i.e., such that $(j,k)\in E(\bfH)$), 
one and only one value $\rg(\theta^o_j-\theta^o_k)$.

Now, by condition (iii), all the terms
\begin{align*}
\omega+ \frac{1}{N}\sum_{j=1,j\neq k}^N \rg(\theta^o_j-\theta^o_k),\quad 1\leq k\leq N
\end{align*}
are equal, so that we can choose $\Omega$ to be equal to any of these terms. This automatically implies that $\Omega t\cdot\mathbf{1}+\theta^o$ is a relative equilibrium of~\eqref{eq:PhaseOscSN}.

\emph{Step~2.} Second, as in~\cite{OroMoeAsh2009}, we linearize~\eqref{eq:PhaseOscSN} at our equilibrium point and check the conditions (i),(ii),(iii) are sufficient conditions (on~$\rg$) to ensure stability. Set $\chi=\theta-(\Omega t \cdot \mathbf{1}+\theta^o)$. We have
\begin{align*}
\dot\chi_k&=\dot\theta_k-\Omega
=\omega +\frac{1}{N}\sum_{j=1,j\neq k}^N \rg(\chi_j-\chi_k+\theta^o_j-\theta^o_k)-\Omega
\end{align*}
which yields the linearized equation $\dot\chi_k=\frac{1}{N}\sum_{j=1,j\neq k}^N \rg'(\theta^o_j-\theta^o_k)(\chi_j-\chi_k)$.
Therefore, by condition (ii), the linearized equation at $\Omega t \cdot \mathbf{1}+\theta^o$ read
\begin{equation}\label{linear}
\dot\chi_k=\frac{1}{N}\sum_{j=1}^N A_{jk}^{\bfH}\,\rg'(\theta^o_j-\theta^o_k)(\chi_j-\chi_k),\quad 1\leq k\leq N.
\end{equation}
Now set $T_{jk}=A_{jk}^{\bfH}\,\rg'(\theta^o_j-\theta^o_k)$ for $j\neq k$ and $T_{kk}=-\sum_{j=1}^N A_{jk}^{\bfH}\rg'(\theta^o_j-\theta^o_k)=- \sum_{j=1}^N T_{jk}$ to write~\eqref{linear} in matrix form
\[\dot\chi_k=\frac{1}{N}\sum_{j=1}^N T_{jk}\chi_j.\]
The stability conditions are given by the spectrum~$\mathfrak{S}(T)$ of the matrix~$T$. In fact~$-T$ is a Laplacian matrix with spectrum
\begin{align*}
\mathfrak{S}(T)=\sset{0,\lambda_2,\dotsc,\lambda_N},
\end{align*}
where the eigenvalue~$0$ corresponds to the direction $\mathbf{1}=(1,\dotsc,1)$ along the group orbit of the phase-shift symmetry.
This means that stability of $\Omega t+\theta^o$ is only determined by the eigenvalues $\lambda_2,\dotsc,\lambda_N$.

Note that none of the eigenvalues $\lambda_2,\dotsc,\lambda_N$ are equal to zero, i.e., zero is a simple eigenvalue of~$T$. Indeed, let~$\tilde\bfH$ denote the graph of which adjacency matrix $A^{\tilde\bfH}$ is defined by $A_{jk}^{\tilde{\bfH}}=T_{jk}$ for $j\neq k$. By condition (ii) we have $A_{jk}^{\tilde{\bfH}} =0$ if and only if $A_{jk}^{{\bfH}}=0$. This means that $E(\bfH)=E(\tilde{\bfH})$ and therefore $\tilde{\bfH}$ has a spanning diverging tree: by Proposition~\ref{prop:agaev2009} the eigenvalue $0$ of the Laplacian matrix of $\tilde{\bfH}$ (i.e., of the matrix $-T$) is simple, which means that $0$ is a simple eigenvalue of $T$.
We also remark that by condition~(ii) and by the Gershgorin circle theorem, all the eigenvalues of $\mathfrak{S}(T)$ belong to the discs centred in $T_{kk}<0$ and of radius $-T_{kk}$.
We can thus conclude that the eigenvalues $\lambda_2,\dotsc,\lambda_N$ have all a negative real part and so $\Omega t \cdot \mathbf{1}+\theta^o$ is an asymptotically stable relative equilibrium.
\end{proof}

In fact Proposition~\ref{prop:realiseattractor} can be generalized to \eqref{eq:PhaseOscCoup} in a similar way:

\begin{cor}
\label{corstable}
Assume that~$\bfH\in \cH(\bfA)$ admits a spanning diverging tree. Then there is a coupling function $\rg$ such that~\eqref{eq:PhaseOscCoup} has an asymptotically stable relative equilibrium $(\Omega t +\theta^o_1,\dotsc ,\Omega t +\theta^o_N)$ satisfying 
${\cG_{\rg,\bfA}(\theta^o)}=\bfH$.
In other words, there exists a coupling function~$\rg$ that stably realises~$\bfH$ for~\eqref{eq:PhaseOscCoup}.
\end{cor}

\begin{proof}
The proof works in an exactly similar way as in Proposition~\ref{prop:realiseattractor}. Under the same conditions (i),(ii),(iii) above, the linearized  equation at $\Omega t \cdot \mathbf{1}+\theta^o$ reads
\begin{align*}\label{linear2}
\dot\chi_k&=\frac{1}{N}\sum_{j=1}^N A_{jk} A_{jk}^{\bfH}\,\rg'(\theta^o_j-\theta^o_k)(\chi_j-\chi_k),\quad 1\leq k\leq N,
\end{align*}
and the coefficients of the corresponding Laplacian matrix~$T$ are
\begin{align*}
T_{jk} &= A_{jk}A_{jk}^{\bfH} \rg'(\theta^o_j-\theta^o_k) &&\text{for } j\neq k,\\
T_{kk} &= - \sum_{j=1}^N A_{jk} A_{jk}^{\bfH}\,\rg'(\theta^o_j-\theta^o_k)=- \sum_{j=1}^N T_{jk}.
\end{align*}
Denoting again by~$\tilde{\bfH}$ the graph with adjacency matrix $A_{jk}^{\tilde{\bfH}}=T_{jk}$ for $j\neq k$, we have that $
A_{jk}^{\tilde{\bfH}} =0$ if and only if $A_{jk}^{{\bfH}} =0$
by the assumptions made for~$\bfH$. Thus, $\tilde{\bfH}$ admits a spanning diverging tree. As in the proof of Proposition~\ref{prop:realiseattractor}, this gives again the asymptotic stability of the relative equilibrium $(\Omega t +\theta^o_1,\dotsc ,\Omega t +\theta^o_N)$.
\end{proof}

We finish with a brief discussion of a sufficient condition for an effective coupling graph~$\bfH$ to be completely stably realised. Suppose that $\theta\in\inte(\Theta_{\rg,\bfA}(\bfH))$ and write the coupled oscillator network~\eqref{eq:PhaseOscCoup} as $\dot\theta = F(\theta)$. Suppose that the boundary~$\partial\Theta_{\rg,\bfA}(\bfH)$ is a locally $N-1$-dimensional semialgebraic set; this is typically the case as discussed in Section~\ref{sec:SkewProducts}. Let~$\bfn(\theta)$ denote the piecewise defined normal vector pointing into the interior. One can now obtain sufficient conditions for~$\bfH$ being invariant by imposing that~$\langle\bfn(\theta), F(\theta)\rangle>0$ for almost all $\theta\in\partial\Theta_{\rg,\bfA}(\bfH)$.

\section{Effective coupling graphs for networks of two and three oscillators}
\label{sec:examples}

For coupling functions with an arbitrary number of dead zones, all effective coupling graphs can be realised as outlined above. But what is the global picture of the dynamics for the minimal case of a coupling function with a single dead zone? In this section we concentrate on this question by exploring small all-to-all coupled phase oscillator networks~\eqref{eq:PhaseOscSN} 
with a coupling function~$\rg\in\cF(1)$. First, we briefly consider a network of $N=2$ oscillators where there are only 4 possible effective coupling graphs. Then we paint the picture for $N=3$ oscillators; there are 64 possible effective coupling graphs and we explore the dynamics numerically.

\subsection{Networks of two oscillators}

One can easily demonstrate that a single dead zone and a single live zone is sufficient to realise all effective coupling graphs for~\eqref{eq:PhaseOscSN} with $N=2$ oscillators. More precisely, choose any $\rg\in\cF(1)$ with $\LZ(\rg)=[-a,2a]$ for $a<\pi/2$, where all inequalities are understood in the interval $[-\pi,\pi]$. Then
\begin{equation*}
\cG_{\rg}(0,c)=
\begin{cases}
\bfK_2 & \text{ if } c\in(-a,a),\\
\bfP_{1,2} & \text{ if } c\in(a,2a),\\
\bfP_{2,1} & \text{ if } c\in(-2a,-a),\\
\bes_2 & \text{ if } c\in(-\pi,-2a)\cup (2a,\pi).
\end{cases}
\end{equation*}
This shows that there is a single coupling function that realises all four subgraphs of~$\bfK_2$. Note that if $\rg$ is dead zone symmetric then only the undirected graph $\bfK_2$ and $\bes_2$ can be realised (cf.~Proposition~\ref{prop:gprops}(iii)).

\subsection{Networks of three oscillators}

\begin{figure}
\centerline{
\raisebox{3cm}{\textbf{(a)}}\hspace{-0.4cm}
\includegraphics[scale=0.34]{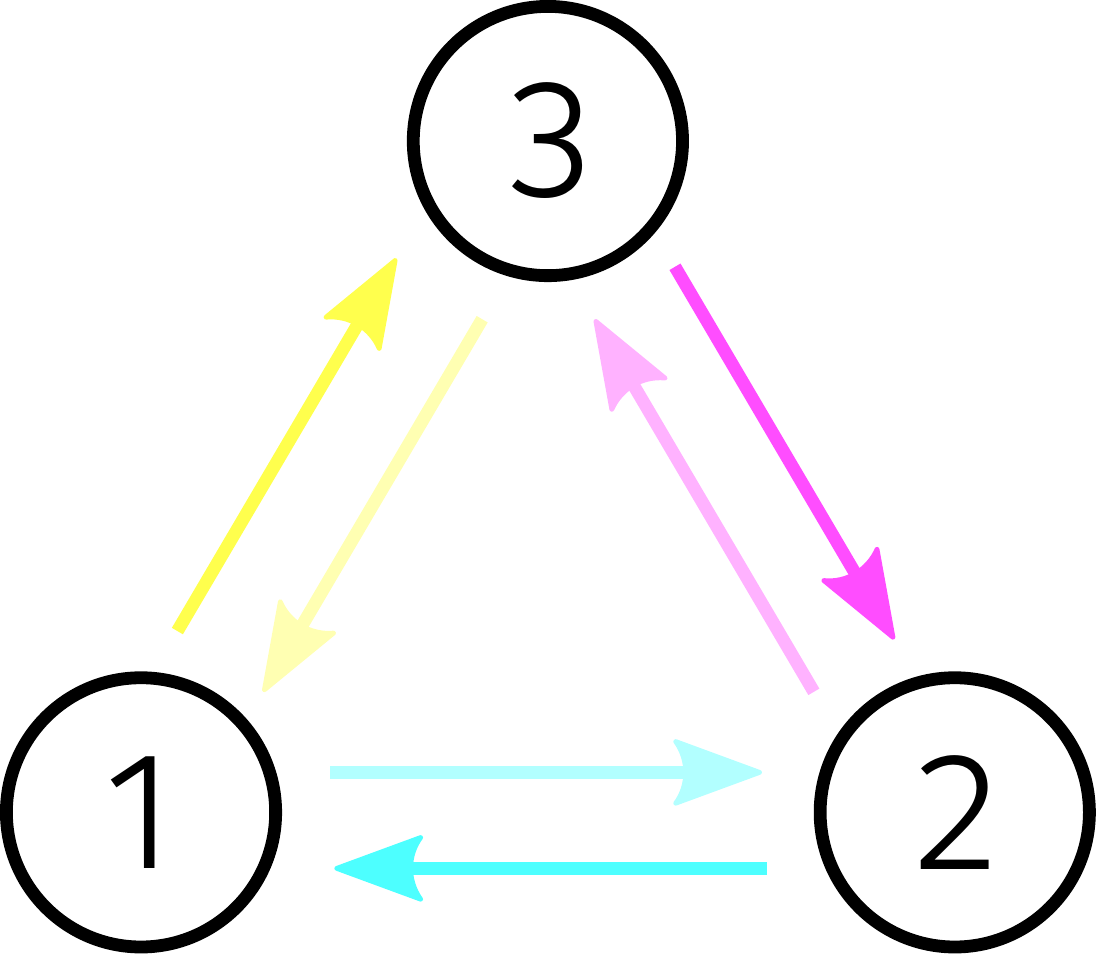}\quad
\raisebox{3cm}{\textbf{(b)}}\hspace{0.2cm}
\includegraphics[scale=0.24]{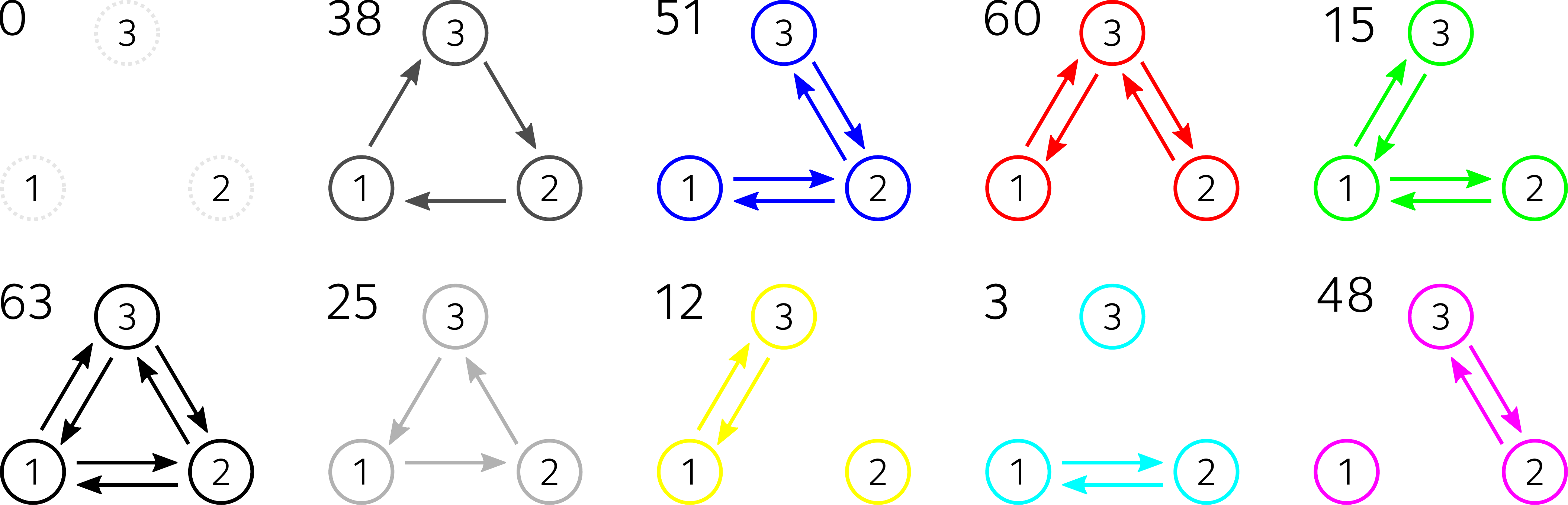}
}
\caption{
We use a colour scheme to identify the graphs in~$\cH_3$. Panel~(a) shows the shades of cyan, magenta, and yellow identified with each directed edge of~$\bfK_3$. If multiple edges are present, the colours are added. Examples of graphs~$\bfH\in\cH_3$ in their associated colours, as well as the corresponding graph numbers~$\nu(\bfH)$, as defined in~\eqref{eq:GN}, are shown in Panel~(b). The subgraphs where all edges to/from a given node are present (and no others) are associated with the colours red, green, and blue. The symmetry that permutes the three nodes acts on the colour scheme by permuting the colour channels. Hence, graphs which are invariant under this symmetry operation have a colour that is invariant under permutation of the colour channels; this includes white for~$\egraph_3$, black for~$\bfK_3$, and shades of gray for the directed cycles $\bfC_{1,2,3},\bfC_{3,2,1}$.
}
\label{fig:ColorExplain}
\end{figure}

We now consider all-to-all coupled networks of $N=3$ oscillators. Since~$\bfK_3$ has~$6$ edges, there are $2^6=64$ different possible effective coupling graphs. By assigning a colour to each edge, we create a scheme that assigns a unique colour to each graph and permutations of the three nodes correspond to permutations of the colour channels. This assignment is outlined in Figure~\ref{fig:ColorExplain} together with examples of graphs coloured in their respective colour. Moreover, the possible effective coupling graphs  can be numbered according to the edges that are present. More specifically, let $\bfH\in\cH_3$ and write~$A^\bfH$ for the associated $3\times 3$ adjacency matrix. Define the \emph{graph number}
\begin{equation}
\nu(\bfH) = A^\bfH_{12}+2A^\bfH_{21}+4A^\bfH_{13}+8A^\bfH_{31}+16A^\bfH_{23}+32A^\bfH_{32}\in\sset{0, \dotsc, 63},
\label{eq:GN}
\end{equation}
which uniquely encodes the realised effective coupling graph as an integer. In particular, we have $\nu(\egraph_3)=0$ and $\nu(\bfK_3) = 63$; more examples are given in Figure~\ref{fig:ColorExplain}.

\begin{figure}
\centering
\raisebox{-7pt}{\includegraphics[scale=1]{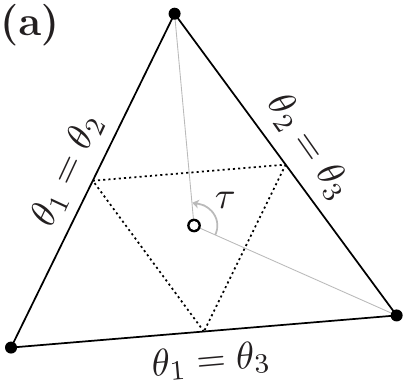}}\qquad\qquad
\includegraphics[scale=1]{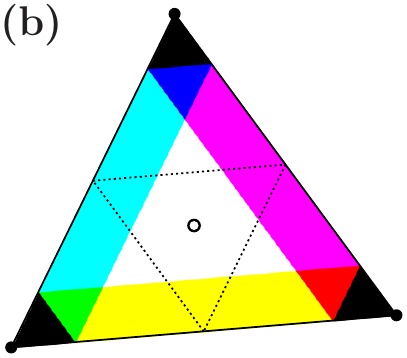}\qquad
\includegraphics[scale=1]{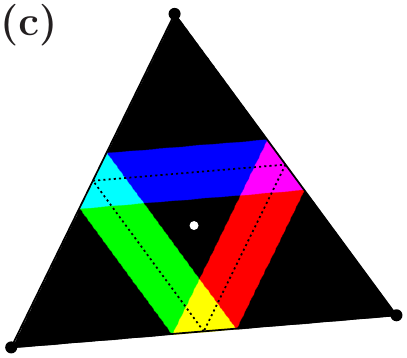}
\caption{
The sets $\Theta_\rg(\bfH)$ for different~$\bfH$ partition the canonical invariant region~$\CIR$ for the fully symmetric system of $N=3$ oscillators. The CIR is sketched in Panel~(a): Its boundary is given by the sets $\theta_1-\theta_2=0$, $\theta_2-\theta_3=0$, and $\theta_3-\theta_1=0$ (black lines) which intersect in~$\Sync$ (black dot,~$\bullet$). The splay phase~$\Splay$ is the centroid (hollow dot, $\circ$) and is the fixed point of the residual~$\Z_3=\langle\tau\rangle$ symmetry which rotates the CIR (indicated by gray lines). Dashed lines indicate phase configurations where one phase difference is equal to~$\pi$.
For a dead zone symmetric coupling function $\rg\in\mathcal{F}(1)$ only the undirected subgraphs of~$\bfK_3$ can be realised; these correspond to the ones shown in Figure~\ref{fig:ColorExplain} excluding the cycles. Panel~(b) shows the partition of the CIR for $\DZ(\rg) = \big(\frac{\pi}{3}, \frac{5\pi}{3}\big)$. Panel~(c) shows the partition for a dead zone symmetric coupling function with $\DZ(\rg) = \big(\frac{5\pi}{6}, \frac{7\pi}{6}\big)$.}
\label{fig:ColoringCIRSym}
\end{figure}

\begin{figure}
\centering
\includegraphics[scale=1]{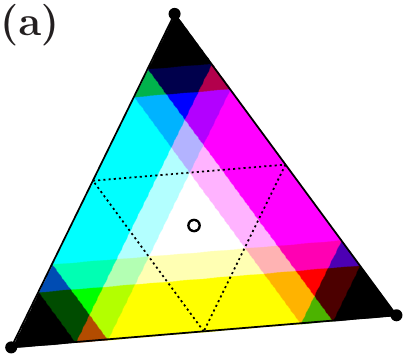}\qquad
\includegraphics[scale=1]{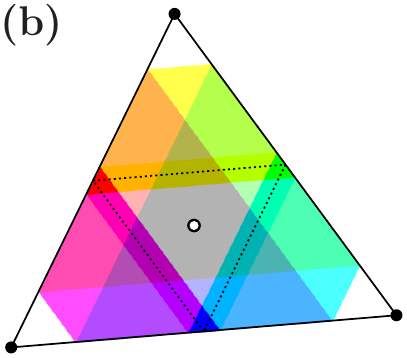}\qquad
\includegraphics[scale=1]{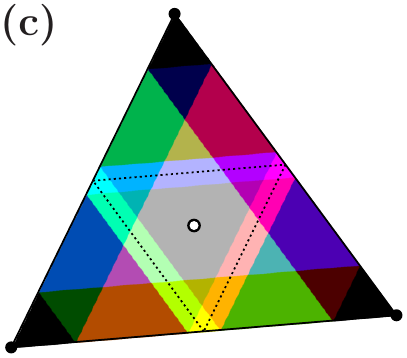}
\caption{
Many effective coupling graphs are possible for $N=3$ oscillators and a general coupling function $\rg\in\mathcal{F}(1)$ with one dead zone. As in Figure~\ref{fig:ColoringCIRSym}, the sets~$\Theta_\rg(\bfH)$ are plotted in the colour corresponding to the effective coupling graph~$\bfH$ in the colour scheme of Figure~\ref{fig:ColorExplain}. We have $\DZ(\rg)=\big(\frac{\pi}{3},\frac{3\pi}{2}\big)$ in Panel~(a), $\DZ(\rg)=\big(-\frac{\pi}{3},\frac{11\pi}{12}\big)$ in Panel~(b) and $\DZ(\rg)=\big(\frac{\pi}{3},\frac{11\pi}{12}\big)$ in Panel~(c).
}
\label{fig:ColoringCIRAsym}
\end{figure}

For a given coupling function~$\rg\in\cF(1)$, the sets $\Theta_\rg(\bfH)$, $\bfH\in\cH_3$, partition the canonical invariant region~$\cC$. To visualize which region of phase space is associated with a given effective coupling graph, we colour the part of~$\cC$ accordingly. 
If~$\rg$ is dead zone symmetric,
the dead zone can be parametrized by a single parameter~$a\in(0,\pi)$ that represents the beginning (or end) of the dead zone. Now suppose that $0\in\LZ(\rg)$ so~$\cG_\rg(\Sync)=\bfK_3$---the case $0\in\DZ(\rg)$ is analogous by ``inverting'' edges (or colours). There are two qualitatively different cases that are shown in Figure~\ref{fig:ColoringCIRSym}: If $0<a<\frac{2\pi}{3}$ then $\cG_\rg(\Splay)=\egraph_3$ and if $\frac{2\pi}{3}<a<\pi$ then $\cG_\rg(\Splay)=\bfK_3$. Many more cases are possible for a general function $\rg\in\cF(1)$; rather than paint a complete picture, we illustrate some cases in Figure~\ref{fig:ColoringCIRAsym}.

The previous considerations were purely in terms of the structure of the effective coupling graph. We now look at examples of the system's dynamics and explore how the effective coupling graph changes along trajectories. To this end, we examine the dynamics of~\eqref{eq:PhaseOscSN} with $N=3$ and the coupling function
\begin{equation}
\label{eq:KSDZ}
\rg(\psi)=-\sin(\psi+\alpha)h(\psi)\quad\text{where }h(\psi)=\frac{1}{2}\left(\tanh(\eps^{-1}(\cos b-\cos(a-\psi))+1\right)
\end{equation}
for constants $a\in[0,2\pi)$, $b\in[0,\pi)$, $\eps>0$ and $\alpha\in [0,2\pi)$. This coupling function is a modulated Kuramoto--Sakaguchi coupling with phase-shift parameter~$\alpha$. We call
\begin{equation}
\DZ^\eps(\rg)=\set{\theta}{\abs{\theta-a}<b}
\label{eq:DZapprox}
\end{equation}
the approximate dead zone of the coupling function~\eqref{eq:KSDZ} since in the limit $\eps\rightarrow 0$ the coupling function~\eqref{eq:KSDZ} has a single dead zone $\DZ(\rg)=\set{\theta}{\abs{\theta-a}<b}$ centred at~$a$ of half-width~$b$; here the inequality is to be understood modulo $2\pi$. In the following we fix $\eps=5\times 10^{-3}$ and $\alpha=1.3$. 

\begin{figure}
\centerline{
\includegraphics[width=14cm]{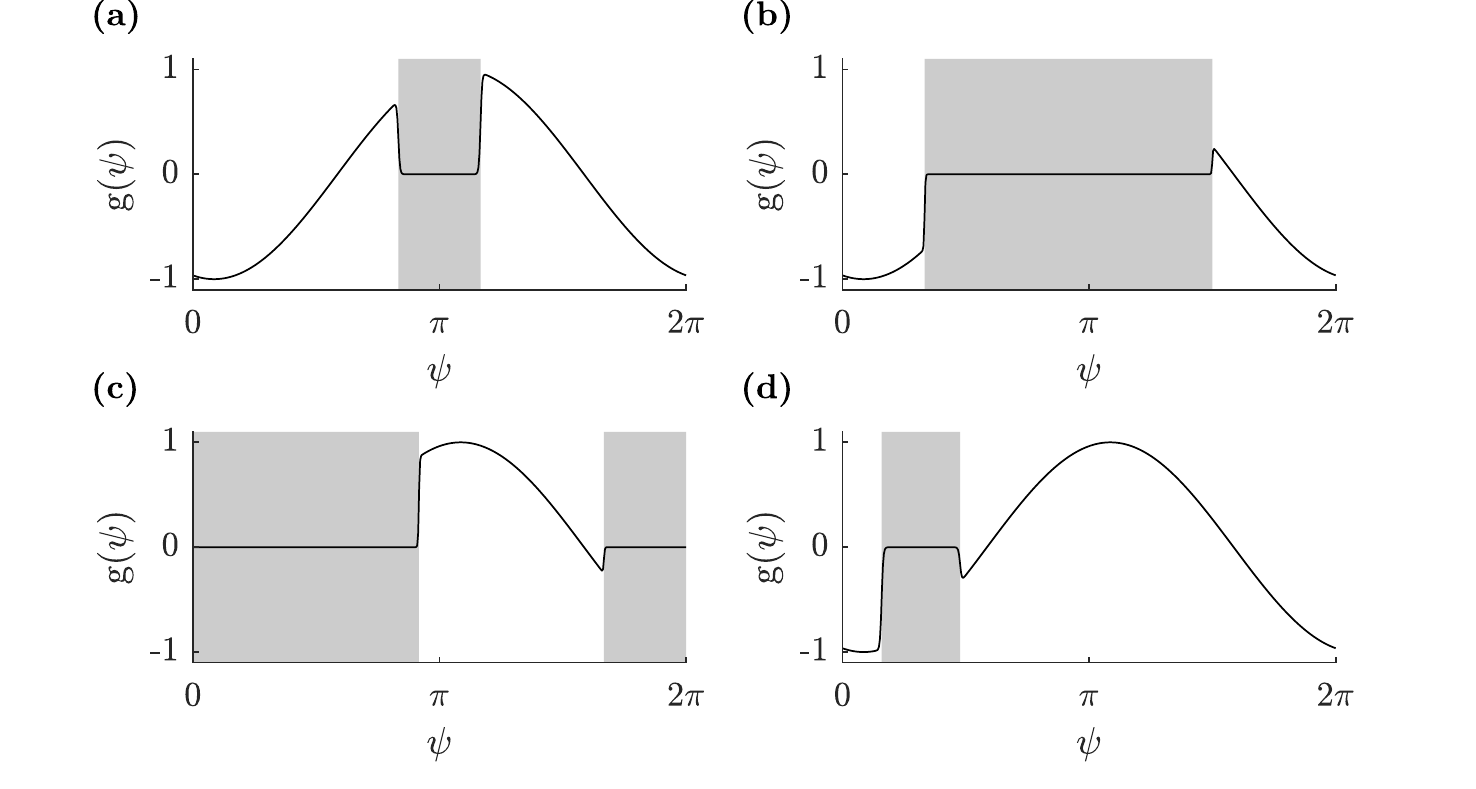}
}
\caption{
The coupling functions~\eqref{eq:KSDZ} provide examples of coupling functions $\rg\in\mathcal{F}(1)$ with one dead zone; here $\eps=5\times 10^{-3}$ and $\alpha=1.3$. The shaded area indicates the dead zone of the coupling function. 
In Panel~(a) we have a dead zone symmetric coupling function with $\DZ(\rg)\approx \big(\frac{5\pi}{6}, \frac{7\pi}{6}\big)$; cf.~Fig~\ref{fig:ColoringCIRSym}(b). 
In Panel~(b) we have $\DZ^\eps(\rg)=\big(\frac{\pi}{3},\frac{3\pi}{2}\big)$; cf.~Fig~\ref{fig:ColoringCIRAsym}(a).
In Panel~(c) we have $\DZ^\eps(\rg)=\big(-\frac{\pi}{3},\frac{11\pi}{12}\big)$; cf.~Fig~\ref{fig:ColoringCIRAsym}(b).
In Panel~(d) we have $\DZ^\eps(\rg)= \big(0.5, 1.5\big)$.}
\label{fig:dzN3f1}
\end{figure}

We explore the dynamics for four examples of coupling functions~\eqref{eq:KSDZ} with approximate dead zones shown in Figure~\ref{fig:dzN3f1}. 
Recall that without dead zones (and any~$N$), the dynamics for Kuramoto--Sakaguchi coupling is well known: Depending on the parameter~$\alpha$, either full synchrony~$\Sync$ or an anti-phase configuration is stable; see also~\cite{Watanabe1994}. 
Now for each of the Kuramoto--Sakaguchi coupling functions with a dead zone, Figure~\ref{fig:dzN3f2} shows a partition of phase space by effective coupling graph (using the same colour scheme as in Figures~\ref{fig:ColoringCIRSym} and~\ref{fig:ColoringCIRAsym}) together with a phase portrait for trajectories of~\eqref{eq:PhaseOscSN} integrated forwards from a grid of initial conditions. Note that the dynamics in~$\Theta_\rg(\egraph_3)$---coloured in white---are trivial and we find, for example, periodic trajectories that visit~$\Theta_\rg(\bfH)$ for multiple~$\bfH\in\cH_3$ as time evolves. 
Such dynamics are impossible for Kuramoto--Sakaguchi coupling without dead zones.

Finally, we consider the entirety of the effective interaction graphs that are realised by a given coupling function. Figure~\ref{fig:dzN3f3} shows the set of realised effective coupling graphs corresponding to the coupling functions in Figure~\ref{fig:dzN3f1} sorted by their graph numbers. We are unable to find a single coupling function~$\rg$ with one live zone that can realise all possible effective coupling graphs---however, we note that a combination of two coupling functions~(for example, (b) and~(c)) suffice to cover all cases.

\begin{figure}
\centerline{
\includegraphics[width=18cm]{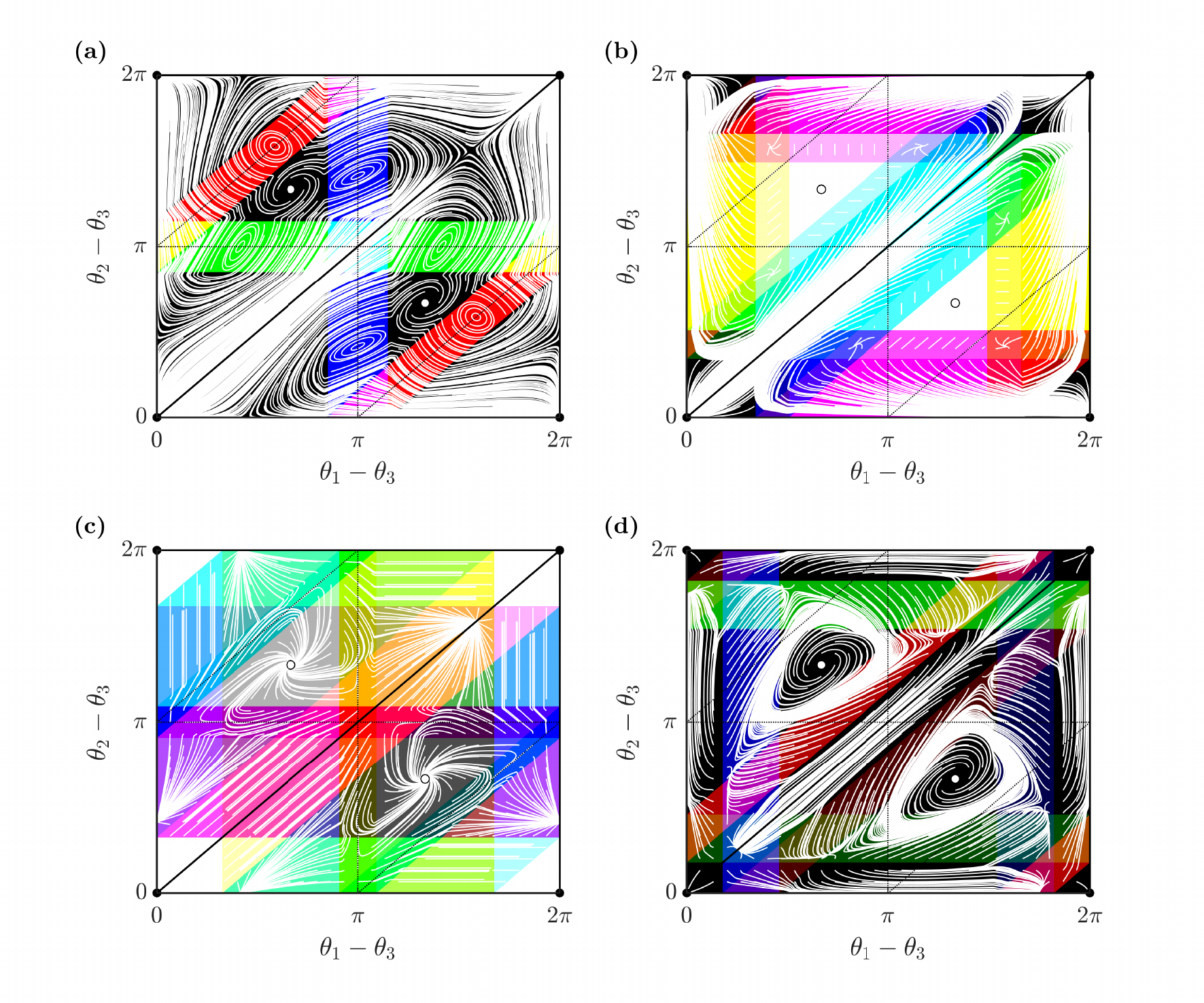}
}
\caption{The phase space for~\eqref{eq:PhaseOscSN} with $N=3$ oscillators and coupling function~$\rg$ with one dead zone as in~\eqref{eq:KSDZ} and parameters as in Figure~\ref{fig:dzN3f1}(a-d) respectively.
As in Figures~\ref{fig:ColoringCIRSym} and~\ref{fig:ColoringCIRAsym}, black lines indicate  the sets $\theta_1-\theta_2=0$, $\theta_2-\theta_3=0$, and $\theta_3-\theta_1=0$ which intersect in~$\Sync$ (black dot, $\bullet$) that bound~$\CIR$ and its symmetric image. The splay phases are indicated by hollow dots ($\circ$) and dashed lines indicate phase configurations where one phase difference is equal to~$\pi$.
As above, the colouring indicates the effective coupling graph overlaid by trajectories started on a regular grid, shown in white---a very wide range of effective coupling graphs are realised; see Figure~\ref{fig:dzN3f3}. For (b) and (c) there are white regions of trivial dynamics where no trajectories are present: these correspond to the effective coupling graph~$\bes_3$. Finally, note that for~(a) and~(c) there are trajectories that visit~$\Theta_\rg(\bfH)$ for multiple~$\bfH\in\cH_3$ as time evolves.
}
\label{fig:dzN3f2}
\end{figure}

\begin{figure}
\centerline{
\includegraphics[width=14cm]{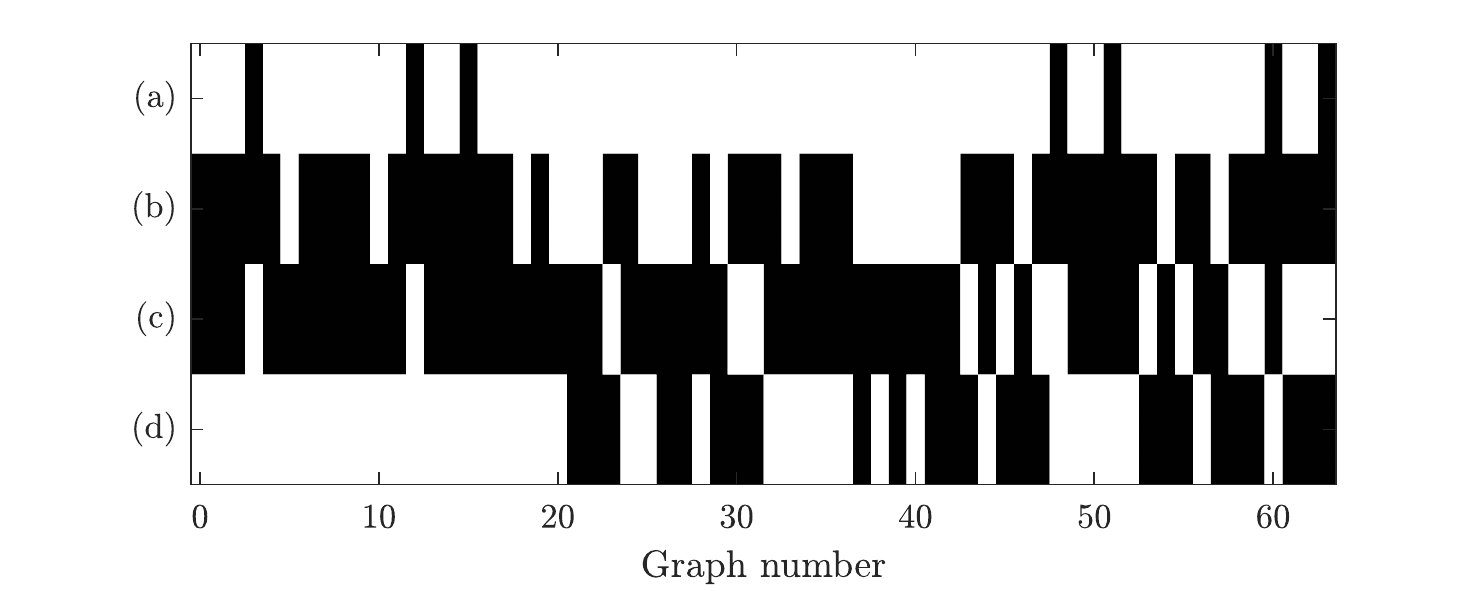}
}
\caption{The possible effective coupling graphs realised using $N=3$ and~\eqref{eq:KSDZ} for parameters as in Figure~\ref{fig:dzN3f1} and some~$\theta$. Black indicates $\Theta_\rg(\bfH)\neq\emptyset$ for~$\bfH\in\cH_3$ with a given graph number, and white indicates $\Theta_\rg(\bfH)=\emptyset$. Since~(a) is a dead zone symmetric coupling function, only undirected subgraphs are realised. By contrast, taking all effective coupling graphs that are realised by coupling functions~(b) and~(c) together, one obtains all possible subgraphs of~$\bfK_3$.}
\label{fig:dzN3f3}
\end{figure}

\section{General observations and discussion}
\label{sec:discuss}

In this paper we have demonstrated that the \emph{effective} coupling graph of a dynamical network is subtly related to network structure, the system state and the presence of dead zones in the interaction. Working with coupled phase oscillator networks~\eqref{eq:PhaseOscCoup}, we give constructions of coupling functions~$\rg$ that achieve any desired subnetwork, possibly using the same $\rg$ (Corollary~\ref{cor:allgraphs}), even in the special and {highly symmetric case} of all-to-all coupling \eqref{eq:PhaseOscSN}. 

In terms of structural questions, we obtain a number of conditions on~$\rg$ and~$\theta$ that guarantee the presence of certain coupling structures in $\cG_{\rg}(\theta)$. There {are} several natural questions that relate to the number, location and lengths of the dead zones to the set of realisable effective coupling graphs. For example, the coupling functions in Figure~\ref{fig:dzN3f1}(b,c) together can realise all possible (embedded) subgraphs of~$\bfK_3$. Two specific questions in this direction for \eqref{eq:PhaseOscSN} are:
\begin{itemize}
\item What is the minimum number $n$ of dead zones such that there is a $\rg\in\cF(n)$ that realises all $\bfH\in\cH_N$?
\item For a fixed number $\ell$ of dead zones, what is minimum number $m$ of coupling functions with~$\ell$ dead zones, 
that between them realise any given $\bfH\in\cH_N$?\footnote{Figure~\ref{fig:dzN3f3} gives evidence that $m\leq2$ for $N=3$ and $\ell=1$.}
\end{itemize}

In terms of the dynamics, probably the most interesting problems relate to how the dynamics of the coupled system interacts with the effective coupling as~$\cG_{\rg}(\varphi_t(\theta^o))$ changes along a trajectory starting at $\theta^o$. This is briefly explored in Section~\ref{sec:transitions} and in the examples in Section~\ref{sec:examples}, but we do not have a complete picture as yet. For example, can one determine which effective coupling graphs can be stably realised, and which can be only transiently realised? What does the passage through effective coupling graphs tell us about the underlying dynamics? How does the partition of phase space into basins of attraction map on to the partition of phase space by $\Theta_{\rg}(\bfH)$?

Here we only considered phase oscillator networks with coupling functions that have simple dead zones, i.e., there are a finite set of nontrivial intervals on which the coupling function vanishes.
This could be developed in three directions: First, one may want to examine coupling with an infinite set of dead zones (though this is likely to be not of much relevance to applications). Second, one could look at the case where the coupling function is locally constant on several intervals where is takes distinct values. Third, one would like to get explicit results for coupling functions with approximate dead zones, i.e., intervals on which the coupling functions are small but nonzero. In this direction, it would be desirable to prove explicit results concerning how well (and over what time scale) networks with dead zones approximate networks where interaction between nodes is small (but non-zero) in parts of phase space.

Finally, we have restricted ourselves here to discussion of these questions for coupled phase oscillators where all interactions are governed by a single periodic phase interaction/coupling function~$\rg$. 
On the one hand, it would be desirable to link the coupling functions considered here to nonlinear oscillator networks through a phase reduction. On the other hand, it would be interesting to explore how these results can be generalised, for example, to more general dynamical systems with pairwise coupling of the form 
\[\dot{x}_k=f(x_k)+\sum_{j\neq k}g(x_j,x_k)\]
for $x_k\in\bbR^d$, where $g(x_j,x_k)$ is null in some open subset of $\bbR^d\times \bbR^d$. Finally, dead zones could also be present in multi-way interactions \cite{Staetal2017, AshBicRod2016}, i.e., interactions where the coupling to~$x_k$ depends simultaneously on the relative position of several of the~$x_j$ with $j\neq k$, {and not only on one of them}.

\subsection*{Acknowledgements}

The authors thank B.~Fernandez, Yu.~Maistrenko and T.~Pereira for useful discussions, and the referees for some very useful comments. PA and CP were funded by EPSRC Centre for Predictive Modelling in Healthcare grant EP/N014391/1. This study did not generate any new data.

\bibliographystyle{unsrt}

\end{document}